\documentclass[11pt]{article}
\usepackage[french, english]{babel}
\usepackage[utf8]{inputenc}
\usepackage{amsmath,amsfonts,amssymb,amsthm,mathrsfs}
\usepackage{minitoc}
\usepackage{bbm}
\renewcommand{\leq}{\leqslant}
\renewcommand{\geq}{\geqslant}
\usepackage[mono=false]{libertine}
\useosf
\usepackage{mathpazo}

\usepackage{wasysym}

\usepackage[dvipsnames]{xcolor}
\usepackage[a4paper,vmargin={3.5cm,3.5cm},hmargin={2cm,2cm}]{geometry}
\usepackage[font=sf, labelfont={sf,bf}, margin=1cm]{caption}
\usepackage{graphicx,graphics}
\usepackage{epsfig}
\usepackage{latexsym}
\usepackage{stmaryrd}
\usepackage{ae,aecompl}

\usepackage[english]{babel}
 \usepackage[colorlinks=true]{hyperref}
\usepackage{pstricks}
\usepackage{enumerate}
\usepackage{tikz}
\usepackage{fontawesome}
\usetikzlibrary{arrows,automata}

\renewcommand{\P}{\mathbb{P}}
\newcommand{\E}{\mathbb{E}}
\DeclareFixedFont{\beaupetit}{T1}{ftp}{b}{n}{2cm}

\newtheorem{theorem}{Theorem}[]

\newtheorem{proposition}[]{Proposition}

\newtheorem{lemma}[]{Lemma}

\theoremstyle{definition}

\newtheorem*{remark}{Remark}

\usepackage{todonotes}

\newcommand{\tr}{\mathbf{t}}

\title{\textsc{Parking on the infinite binary tree}}
\author{
David \textsc{Aldous}\thanks{U.C. Berkeley.\hfill  \href{mailto:}{\texttt{aldousdj@berkeley.edu}}}\qquad\&\qquad
Alice \textsc{Contat}\thanks{Universit\'e Paris-Saclay.\hfill  \href{mailto:alice.contat@universite-paris-saclay.fr}{\texttt{alice.contat@universite-paris-saclay.fr}}}
\qquad\&\qquad
Nicolas \textsc{Curien}\thanks{Universit\'e Paris-Saclay.\hfill  \href{mailto:nicolas.curien@gmail.com}{\texttt{nicolas.curien@gmail.com}}}
\qquad\&\qquad
Olivier \textsc{H\'enard}\thanks{Universit\'e Paris-Saclay.\hfill  \href{mailto:olivier.henard@universite-paris-saclay.fr}{\texttt{olivier.henard@universite-paris-saclay.fr}}}}

\date{}
\begin{document}
\maketitle 

\begin{abstract}Let $(A_u : u \in \mathbb{B})$ be i.i.d.~non-negative integers that we interpret as car arrivals on the vertices of the full binary tree $ \mathbb{B}$. Each car tries to park on its arrival node, but if it is already occupied,  it drives towards the root and parks on the first available spot. It is known \cite{GP19,BBJ19} that the parking process on $ \mathbb{B}$ exhibits a phase transition in the sense that either a finite number of cars do not manage to park in expectation (subcritical regime) or all vertices of the tree contain a car and infinitely many cars do not manage to park (supercritical regime).  We  characterize those regimes in terms of the law of $A$ in an explicit way. We also study in detail the critical regime as well as the phase transition which turns out to be ``discontinuous''.%give explicit necessary and sufficient conditions on the law $X$ for which the number of cars that do not manage is finite  (subcritical regime)    be Suppose that on each vertex of the infinite binary tree a random number $X_u$ of car arrive 
\end{abstract}

\begin{figure}[!h]
 \begin{center}
 \includegraphics[height=5cm]{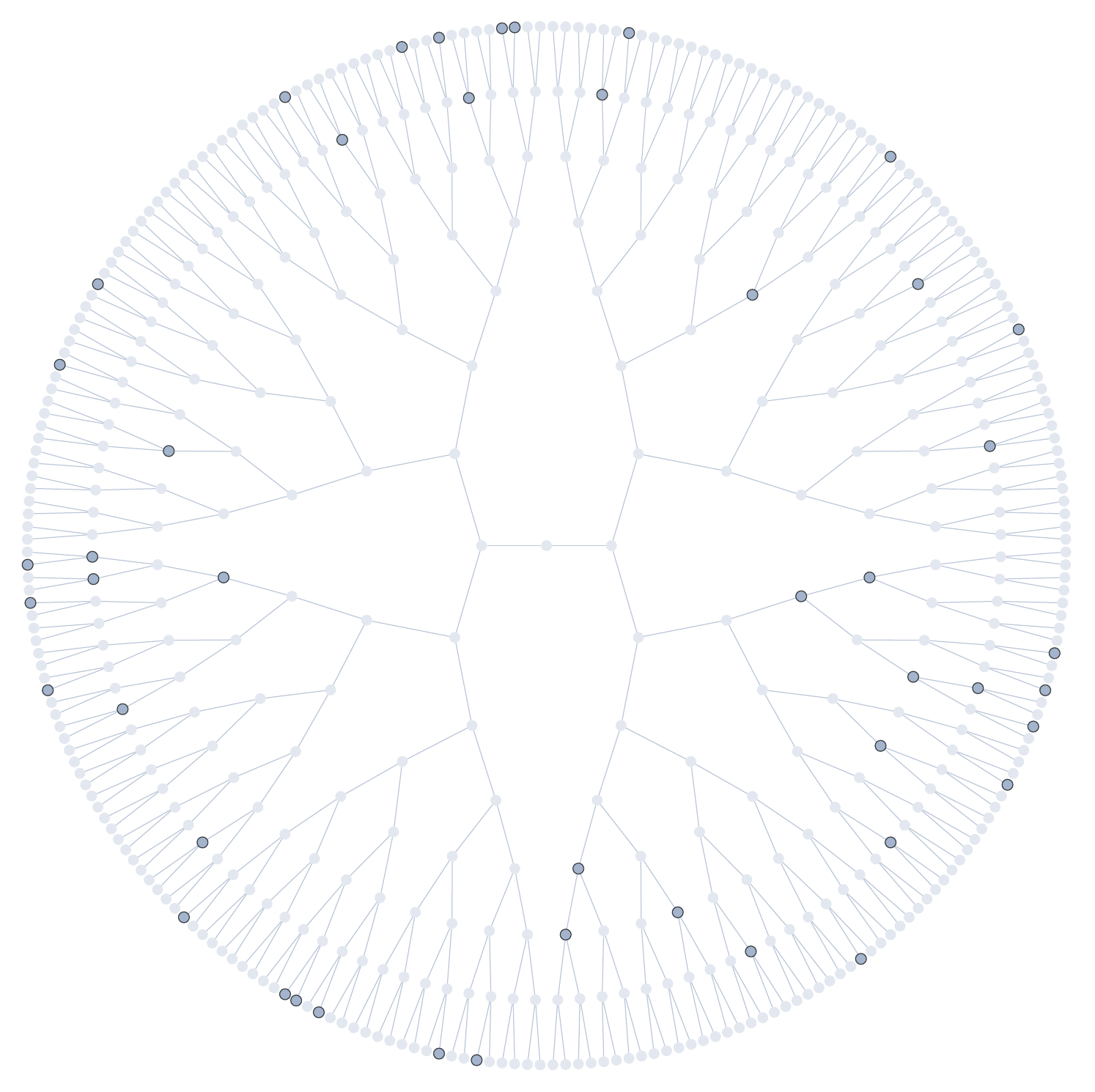}
   \includegraphics[height=5cm]{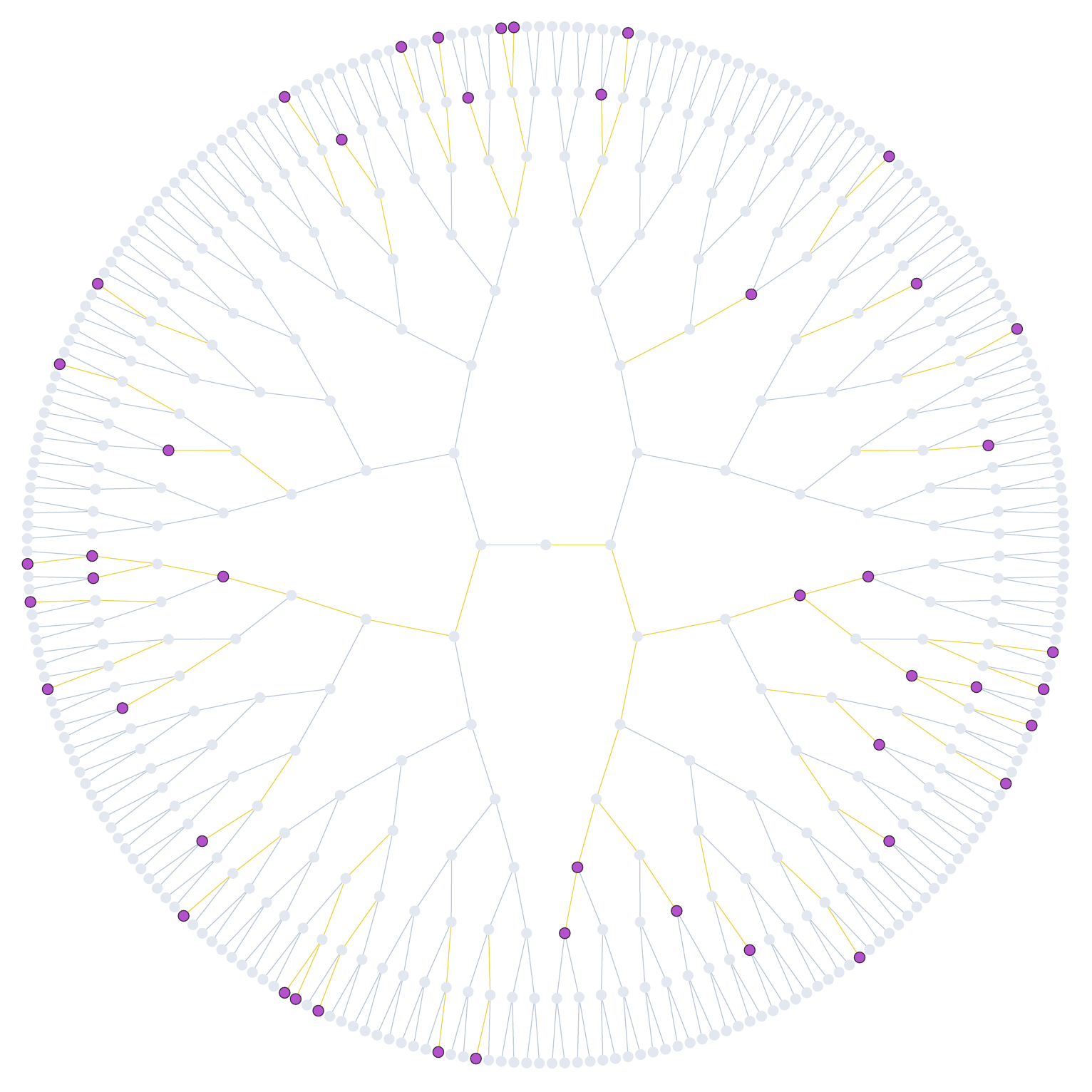}
    \includegraphics[height=5cm]{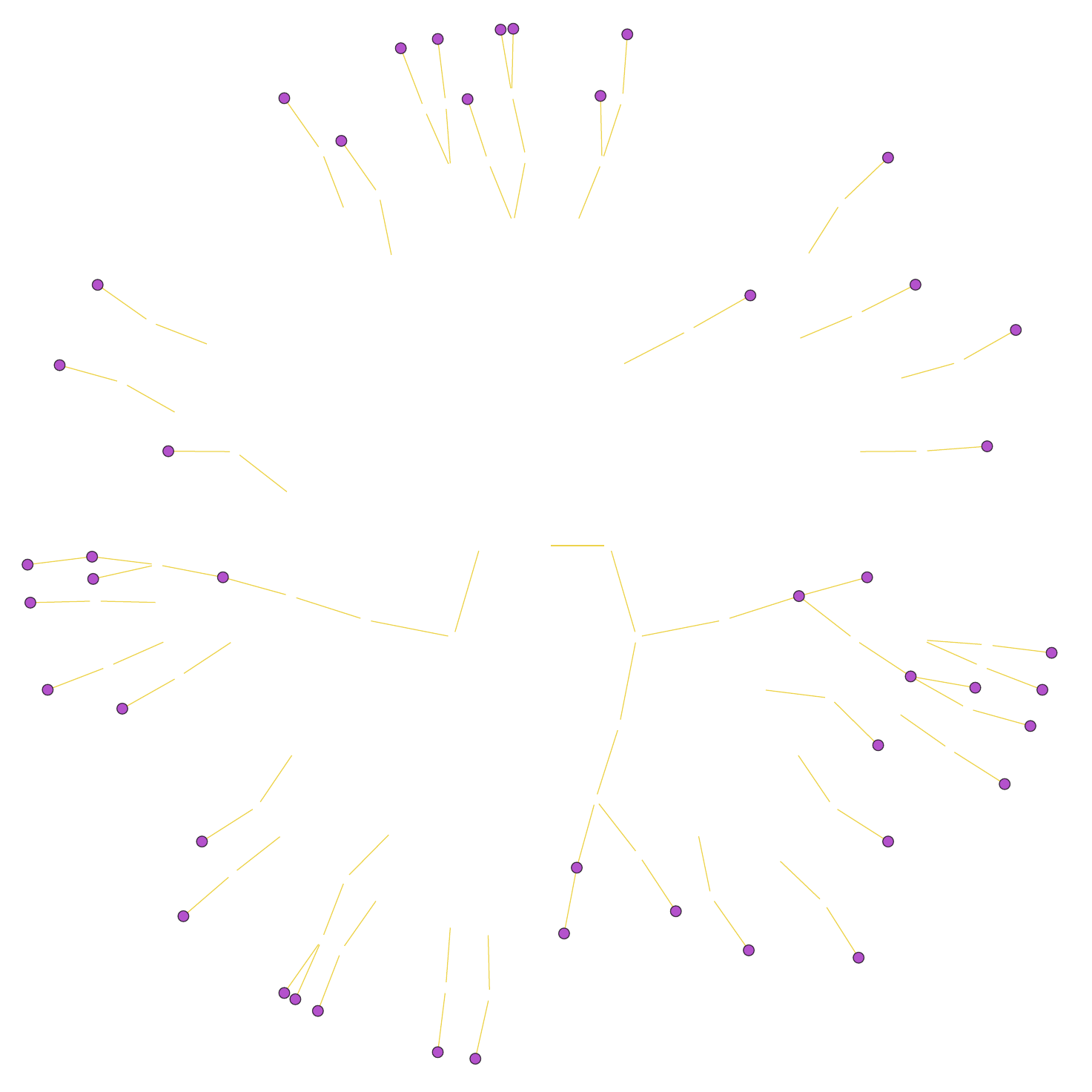}
 \caption{Simulation of the parking process on the first $9$ levels of the full binary tree (the edges are oriented towards the origin of the tree which is at the center of the figures). The dotted vertices on the left figure initially contain $2$ cars whereas the others are void. In the middle and right figures, the highlighted edges have seen a positive flux of car.}
 \end{center}
 \end{figure}

\section {Introduction}
The parking process is a central algorithm in combinatorics and probability. When the underlying graph is an oriented line, it was first studied by Konheim \& Weiss \cite{konheim1966occupancy} in relation with hash tables and it has led to many developments in probability notably via connections with the Brownian continuum random tree and the additive coalescent \cite{chassaing2002phase}. Recently, Lackner \& Panholzer \cite{LaP16} started the systematic study of the parking functions on finite rooted trees. This triggered an intense activity on the model of parking on a random \emph{critical} Galton--Watson tree. In particular, a phase transition was proved to occur and the threshold was located in an increasing level of generality \cite{GP19,CH19,contat2020sharpness}. Furthermore a surprising connection with the Erd\"os--R\'enyi random graph and the multiplicative coalescent was unraveled in \cite{ConCurParking}. 

However, much less is known about the parking scheme on \textbf{supercritical} Galton--Watson trees, apart from the existence of a phase transition \cite{GP19,BBJ19} and despite an intense activity on the closely related Derrida--Retaux model \cite{chen2019derrida}. The goal of this paper is to close this gap and locate and study the phase transition in the case of the parking process on the infinite binary tree (see Section \ref{sec:comments} for extensions).

\paragraph{Parking on the infinite binary tree.} Consider the full planar rooted binary tree. Its vertices can be conveniently represented by the finite words on two letters $ \mathbb{B}  = \cup_{n \geq 0} \{0,1\}^n$, with $\{0,1\}^0 = \varnothing$ being the root of the tree and with edges between the words $u$ and $u0$ and the words $u$ and $u1$. Those vertices  will be interpreted as free parking spots, each spot accommodating at most $1$ car. On top of that tree, we consider a non-negative integer labeling  $(A_u : u \in \mathbb{B})$  representing the number of cars arriving on each vertex $u \in \mathbb{B}$. Each car tries to park on its arrival vertex, and if the spot is occupied, it travels downwards in direction of the root of the tree until it finds an empty vertex to park. If there is no such vertex on the path towards the root $ \varnothing$, the car exits the tree, contributing to the \emph{flux} of cars at the root. If we introduce the random variable $$X := \mbox{ number of cars which have visited } \varnothing,$$ the outgoing flux of cars is then simply  $\mathrm{F} = (X-1)_{+} = \max(X-1,0)$. As we will see in Section \ref{sec:parkingpropre}, the final configuration (flux and status void/occupied for the vertices), and in particular the value of variables $X$ and $ \mathrm{F}$, does not depend  upon the order chosen to park the cars.

\begin{figure}[!h]
 \begin{center}
 \includegraphics[width=14cm]{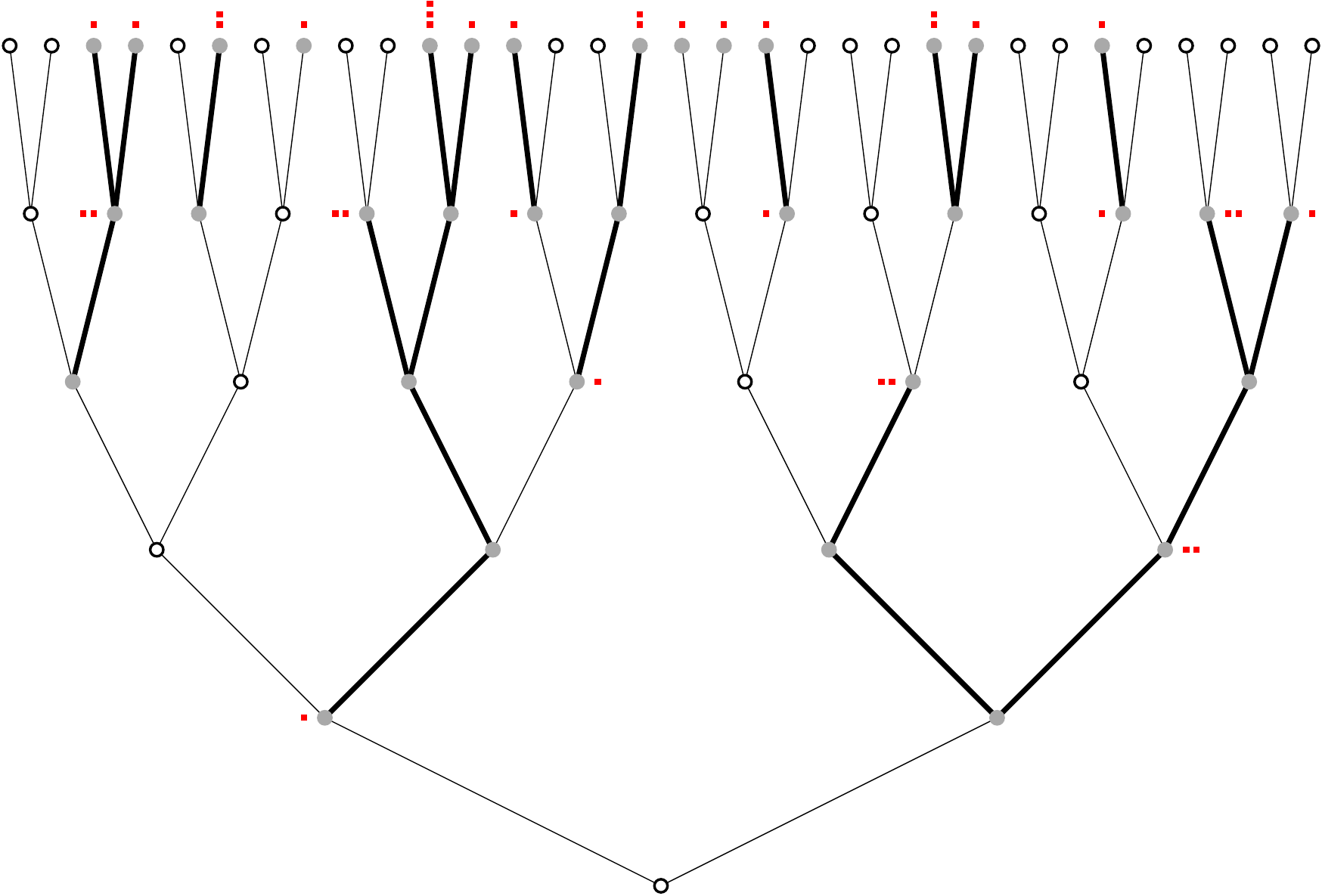}
 \caption{Illustration of the parking process in the first $5$ levels of the full binary tree. The car arrivals are represented by red squares (including the cars that may come from higher levels on top of the tree). After the parking process, the vertices accommodating a car are displayed in gray, whereas the free spots are displayed in white. The connected components of parked cars are drawn with thick lines, they are  \emph{fully parked trees}. \label{fig:parkingbinary}}
 \end{center}
 \end{figure}

In the remainder of this paper we shall suppose that the car arrivals $(A_u : u \in \mathbb{B})$ are i.i.d.~with a given distribution $\mu = (\mu_{k} : k \geq 0)$ on $\{0,1,2,3,...\}$. To avoid trivialities, we always suppose that $\mu(\{0,1\}) < 1$ for otherwise the cars would always park on their arrival node. We let $$G(x) = \sum_{k \geq 0} \mu_{k} x^k$$ be the generating function of the law $\mu$. One can then establish a dichotomy (see Lemma \ref{lem:sub/super} and also \cite{GP19,BBJ19} as well as  Proposition  \ref{prop:noblack} below):
\begin{itemize}
\item Either the number $X$ of cars that visited the root $\varnothing$ has a finite mean and all clusters of parked vertices are finite almost surely, we call this phase the  \textbf{subcritical} regime,
\item Or almost surely $X= \infty$ and actually, all vertices of $ \mathbb{B}$ are occupied after the parking process, we call this phase the \textbf{supercritical} regime.
\end{itemize}
We shall furthermore distinguish the \textbf{critical} regime, when it is not possible to stochastically increase $\mu$ and stay subcritical. A first trivial remark is that when $ \mathbb{E}[A]>1$ the process is necessarily supercritical (since there are more cars than parking spots on average). Our main result is then a characterization of those regimes explicitly in terms of the generating function $G$ of $\mu$:
\begin{theorem}[Location of the phase transition] \label{thm:phase} Suppose that there exists  $t_c \in (0, \infty)$ such that %The parking process is subcritical if and only if 
$$t_c = \min \{ t \geq 0:\ 2(G(t)- t G'(t))^2 = t^2 G(t) G''(t) \} \qquad (\star).$$
Then the parking process is subcritical if and only if 
 \begin{eqnarray} (t_c-2)G(t_c) \geq t_c(t_c-1)G'(t_c). \label{eq:threshold}\end{eqnarray} 
\end{theorem}

The condition $(\star)$ on the existence of $t_c$ is mild and is for example verified for all generating functions with infinite radius of convergence (see Remark \ref{rek:RCVinfty}). When $(\star)$ is not verified, we provide a method to check if we are in the subcritical phase, see Section \ref{sec:non-generic}. Checking the signs of the two sides of the inequality \eqref{eq:threshold}, see Remark \ref{tcgeq2}, we see in particular that the generating function $G$ must have a radius of convergence at least $2$ to be in the subcritical phase. This can easily be explained by probabilistic arguments: otherwise, the maximum of $2^{n}$ independent copies of a random variable with law $\mu$ is larger than $n$ with high probability, so that the car arrivals at the single level $n$ of $ \mathbb{B}$ suffice to guarantee that the root $\varnothing$ is occupied,  see Lemma \ref{lem:sub/super}. 
The same argument actually even proves that the radius of convergence of $X$ (which is stochastically larger than $A$) must stay above $2$ in the subcritical regime. 
Notice that deciding whether $ \mu$ is subcritical for parking depends in a subtle way on the distribution as opposed to the case of critical Galton--Watson trees \cite{GP19,CH19,contat2020sharpness} where its depends only on the first two moments.

 Let us  give a couple of examples of application of our theorem in  the case of a car arrival distribution that is parametrized by a family $(\mu_{\alpha} : 0 \leq \alpha)$ which is  stochastically increasing with mean $\alpha \geq 0$: in this case the parking is subcritical if $ \alpha \leq \alpha_{c}$ and supercritical if $\alpha >\alpha_{c}$, for some threshold $ \alpha_{c}$ depending on the family of laws:  \begin{itemize}
\item \textit{Binary$_{\, 0/2}$ arrivals.} If $\mu_{\alpha} = (1- \frac{\alpha}{2}) \delta_0 + \frac{\alpha}{2} \delta_2$, then the critical threshold is $$ \alpha_{c}( \mathrm{Binary_{\,0/2}}) = \frac{1}{14} \cdot$$ This settles an open problem of Bahl, Barnet \& Junge \cite{BBJ19}. Obviously the value of $ \alpha_{c}$ is in agreement with the bounds $ \frac{1}{32}\leq \alpha_c \leq \frac{1}{2} $ of \cite[Proposition 3.5]{GP19} and improved to $ 0.03175\leq \alpha_c \leq 0.08698$ in \cite[Proposition 4]{BBJ19}.
\item \textit{Binary$_{\, 0/k}$ arrivals.} More generally if $\mu_{\alpha} = (1- \frac{\alpha}{k}) \delta_0 + \frac{\alpha}{k} \delta_k$ for some $k \in \{2,3,4,...\}$, then the threshold is
$$ \alpha_c( \mathrm{Binary_{\, 0/k}}) = \frac{k}{1 + 2^{-k-2} \left( 3 + \sqrt{ \frac{k+7}{k-1}}\right)^k \left( (k-1)(k+4)+ k \sqrt{(k+7)(k-1)}\right)} .$$
\item \textit{Poisson arrivals.} If $\mu_{\alpha}$ is Poisson with mean $ \alpha$, then 
$$  \alpha_{c}( \mathrm{Poisson}) = 3 - 2 \sqrt{2}.$$
\item \textit{Geometric arrivals.} If $\mu_{\alpha} = p^{k}(1-p)$ for $k \geq 0$ is a geometric law with mean $ \alpha = \frac{p}{1-p}$ then $p_{c}(\mathrm{Geometric}) = 1/9$ and 
$$ \alpha_{c}( \mathrm{Geometric}) = \frac{1}{8} \cdot$$
\end{itemize}

\paragraph{The critical regime.} Let us now focus more precisely on the critical regime : we assume that $\mu$ is subcritical (and $(\star)$ holds) and that it is not possible to stochastically increase $\mu$ while remaining subcritical. As we shall see in Section \ref{sec:proba}, this means that the inequality in \eqref{eq:threshold} is actually an equality. Recall that  we denoted by $X$ the number of cars that \emph{visited} the root $ \varnothing$ of $ \mathbb{B}$ during the parking process. We set for $k \geq 0$, $p_{k} = \mathbb{P}(X=k)$ and will use the shorthands $p_{\circ} =p_{0}$ and $p_{\bullet} = p_1$ for respectively the probability that the root is void and the probability that the root is at the bottom of a parked cluster without flux, after parking. In the following, we shall call \emph{white} the clusters of void vertices, and \emph{black} the clusters of parked vertices.

\begin{theorem}[Critical computations] \label{thm:critical} Suppose that $(\star)$ holds and that \eqref{eq:threshold} is an equality. Then almost surely the root $\varnothing$ is void or it belongs to a finite black cluster, and we have
$$ p_{\circ} = \frac{t_c^2}{4(t_c-1)G(t_c)}  \quad \mbox{ and } \quad  p_{\bullet} = \sqrt{ \frac{p_{\circ}}{ \mu_0}} - p_{\circ}.$$
%Furthermore, the radius of convergence of $X$ is equal to \nico{C'est toi qui sait. Mais il vaut mieux que ce soit plus grand que $1$, voire meme $2$ ;)} \alice{Je rente ma chance : $t_c$. D'ailleurs, c'est peut-\^etre aussi $t_c$ en sous critique mais je suis pas sure} \nico{On va voir ce qu'on met... j'hesite encore}
\end{theorem}
These calculations have a few surprising consequences:
\begin{itemize}
\item $ \mathbb{E}[X]<\infty$. The fact that the expectation of the flux of cars is finite in the whole subcritical regime (including at criticality) may be surprising at first, but this can actually be seen from the recursive distributional equation satisfied by $X$ by splitting at the root of $ \mathbb{B}$
 \begin{eqnarray} \label{eq:recursive} X & \overset{(d)}{=}& \quad  (X_{1}-1)_{+} + (X_{2}-1)_{+} + A,  \end{eqnarray} where $X_{1},X_{2}$ are two copies of law $X$ independent of the car arrivals $A$ of law $\mu$. Indeed, the RHS has expectation at least $  2\mathbb{E}[X]-2 $  which is strictly larger than $ \mathbb{E}[X]$ as soon as $\mathbb{E}[X] >2$. Iterating the argument, one sees that there is no a.s.~finite solution to the above recursive distributional equation which has a mean $>2$, see \cite[Theorem 1.1]{BBJ19} for details. Actually, as we already mentioned the variable $X$ must have a radius of convergence larger than $2$, even at the critical point, see the forthcoming Lemma \ref{lem:sub/super}. Also, plugging the value of $p_{\circ} = \mathbb{P}(X=0)$ into \eqref{eq:recursive} we  deduce that 
 $$  \mathbb{E}[X] + \mathbb{E}[A]= 2(1- p_{ \circ}),  \qquad \text{ or equivalently} \qquad \mathbb{E}[\mathrm{F}] + \mathbb{E}[A]= 1-p_{ \circ} $$
Now, by Remark \ref{ref:p0}, on the subcritical regime we have $p_{\circ} > \frac{1}{2}$, so that the LHS of the left identity is bounded by 1 and the LHS of the right identity by $1/2$; one can also show that these bounds are sharp.
%\olivier{un peu bizarre de mettre ca ici dans la  partie critique ? en fait la remarque $p_{\circ} > \frac{1}{2}$ pourrait etre un corollaire a part entiere, ou alors etre incluse dans le thm (?) je pense} \nico{On pourrait, mais les formules sont plus jolies en critique et puis ce qui est vrai en critique l'est facilement en sous-critique}
\item  It will follow from our combinatorial decomposition that the clusters of void vertices are actually Bienaym\'e--Galton--Watson trees with offspring distribution $\xi$ given by 
 $$   \mathbb{P}( \xi=0) = \frac{p_{\bullet}^{2}}{ (p_{\circ} + p_\bullet)^2}, \quad \mathbb{P}( \xi=1) = \frac{2p_{\bullet}p _{\circ}}{ (p_{\circ} + p_\bullet)^2}, \quad \mathbb{P}( \xi=2)=\frac{p_{\circ}^{2}}{ (p_{\circ} + p_\bullet)^2} .$$
Again, since by Remark \ref{ref:p0} we have $p_{\circ} > \frac{1}{2}$, those trees are \textbf{supercritical}, implying that at criticality there are (infinitely many) infinite white clusters. On the contrary, we shall see in Proposition \ref{prop:noblack} that in the subcritical regime (including the critical case), there are no infinite black clusters.
% What is maybe even more surprising is that $X$ still has exponential decay at criticality! We shall give a heuristic explanation of this fact below. Another consequence of Theorem \ref{thm:critical} is that since $p_{\circ} >1/2$, although all clusters of parked cars are finite, they are infinitely many infinite clusters of void vertices. 
 \end{itemize}
 Those phenomena underline the fact that the phase transition in the parking process is \emph{discontinuous} contrary e.g.~to the case of Bernoulli percolation on $ \mathbb{Z}^{2}$.
 
 \paragraph{Fully parked trees and their enumeration.} The proofs of our main results rely on a simple combinatorial decomposition into clusters of parked vertices and the enumeration thereof. More precisely, a fully parked tree $ \mathbf{f}$ is a subtree of $ \mathbb{B}$ containing the root, decorated with car arrivals, so that all those cars manage to park on $ \mathbf{f}$ and that reciprocally all vertices of $ \mathbf{f}$ are parked. If $F\equiv F_{\mu}(x)$ is the generating function of fully parked trees counted with a weight $x$ per vertex and incorporating the $\mu$-weight of car arrivals, see Section \ref{def-fpt} for the precise definition, then the high-level idea of the proof is to write the fixed point equations for $p_{\circ}$ and $p_{\bullet}$, which are
   \begin{eqnarray*} 
   p_{\circ} = \mu_{0}( p_{\circ} + p_{\bullet})^{2}\quad \mbox{ and } \quad 
   p_{\bullet} = p_{\circ}F_{\mu}( p_{\circ}),
    \end{eqnarray*}
and translate the idea that we decompose the structure into the (finite) clusters of parked vertices.  Theorem \ref{thm:phase} boils down to deciding whether we have a non trivial solution $p_{\circ} \ne 0$ to these equations (otherwise we are in the supercritical regime). The critical regime corresponds to the case where $p_{\circ}$ is exactly the radius of convergence of $F$. Thus the main ingredient in the proof is the ``computation'' of the generating function $F$. The enumeration of fully parked trees has already been considered in the combinatorics literature \cite{LaP16,chen2021enumeration,king2019prime,ConCurParking,contat2022last} and it shares many similarities with the enumeration of planar maps.  The idea is to enumerate a more complicated structure, namely fully parked trees with a possible flux of cars at the origin. Those are defined as fully parked trees, except that now the number of cars may be larger than the number of vertices of the tree so that the number of cars $X$ visiting the root of the tree may be strictly larger than $1$.  If $F \equiv F_{\mu}(x,y)$ is the generating function of fully parked trees with weight $x$ per vertex and $y$ per outgoing car, then writing a recursive decomposition at the root vertex we obtain
 \begin{eqnarray} \label{eq:tutteintro}
    F(x,y) =  \frac{x}{y}\left( \big(1+F(x,y)\big)^2 G(y)     - \big(1+F(x,0)\big)^2 G(0) \right).   \end{eqnarray}
   % justif : un fpt non vide est la concatenation de deux fpt possiblement vides et d'une arrivee de voitures a la racine, sauf dans le cas ou les flux des deux fpt sont nuls et aucune voiture n'arrive a la racine}

  These equations are reminiscent of Tutte's equation \cite{Tut62} in the realm of planar maps where the perimeter of the external face plays the role of our outgoing flux of cars. In this equation, the variable $y$ is called the \emph{catalytic variable} since its role is to disappear to recover $F(x,0)  \equiv F(x)$,  the generating function of fully parked trees with no flux. We apply the standard kernel method \cite{BMJ06} to solve those equations, see Section \ref{sec:enumeration} for details. \\

Once we have sufficient information on $F$, the proofs of our main results are rather straightforward. Deciding whether $p_{\circ}$ is non trivial boils down to an inequality on $F(x_{c},0)$ at its radius of convergence $x_{c}$, see Proposition \ref{prop:criti}. Under the assumption $(\star)$, this inequality is equivalent to \eqref{eq:threshold} and the critical case corresponds to the case when $p_{\circ}$ coincides with the radius of convergence of $F(\cdot,0)$. Furthermore, in the subcritical case the generating function of the outgoing flux of cars is given by $ p_{\circ}F(p_{\circ},y)$  (see \eqref{eq:funda1}).

\paragraph{Growth-fragmentation trees.} It will follow from our decomposition that conditionally on $X=1$, i.e.~on $\varnothing$ being the  root of a fully parked tree with no flux, then the cluster of parked cars above $\varnothing$ is a random fully parked tree whose size has generating function $ F(p_{\circ} z)/F(p_{\circ})$. In the critical regime, since $p_{\circ}$ corresponds to the radius of convergence of $F$,  the tail of the cluster size has a subexponential decay and in the generic situation (e.g.~when the car arrivals have bounded support), we actually have 
  \begin{eqnarray} \mathbb{P}( \varnothing \mbox{ is a the root of a parked cluster of size } n) \sim \mathrm{cst} \cdot n^{-5/2},   \label{eq:5/2}\end{eqnarray}  the exponent $5/2$ being common in the theory of map enumeration. Furthermore, we also believe that in the generic situation, rescaled large fully parked trees converge after normalization towards the growth-fragmentation trees that already appeared in the study of scaling limits of random planar maps and the Brownian sphere, see  \cite{BBCK18,BCK18,le2020growth} or \cite[Chapter 14.3.2]{CurStFlour}. We already made a similar conjecture for the scaling limits of parked components in the  parking process on large uniform Cayley trees \cite[Conjecture 1]{ConCurParking}. It is interesting to notice that although the phase transition in the parking on $ \mathbb{B}$ is of a different flavor (the phase transition in the case of critical Galton--Watson trees is ``continuous''), the large scale geometry of the critical components should be the same. However, there are non-generic situations (with specific car arrivals distributions having heavy-tail) where  \eqref{eq:5/2} does not hold and where we expect different scaling limits. See Section \ref{sec:comments} and \cite{chen2021enumeration} for a similar phenomenon in the case of enumeration of non-binary plane fully parked trees. We plan to address those questions in following works.\medskip 

\noindent  \textbf{Acknowledgments.} A.C and N.C. acknowledge the support from ERC 740943 GeoBrown. Part of this work was initiated during a conference in CIRM and we thank our host for its hospitality.

\section{Background}
In this section we formally present the parking process on $ \mathbb{B}$ and gather a few ``rough'' probabilistic results (mostly adapted from \cite{GP19,BBJ19}). 
\subsection{Parking on infinite trees}
%\paragraph{Definition}
\label{sec:parkingpropre}
Let $\tau$ be a rooted locally finite (plane) tree decorated with car arrivals $(a_u : u \in \tau)$. As described in the introduction, cars try to park on their arrival node, and if the spot is taken they travel downwards in search of the first empty spot and, in case there is no such spot, exit at the root. In the case $\tau$ is finite, an easy Abelian property shows that the number of cars visiting each vertex of the tree does not depend on the order in which we park the cars, see Section 2.1 of \cite{LaP16}.

On infinite trees, to prevent cumbersome issues, %since to decide whether a given parking spot is empty or not, one may have to ``climb-up'' tree along an infinite ray. In presence of such a situation, the parking is not ``local'', see Figure \ref{fig:non-local}.  \begin{figure}[!h]
% \begin{center}
% \includegraphics[width=12cm]{nonlocality.pdf}
% \caption{A problematic situation in the parking process on infinite tree: On the left, the starting configuration with cars arriving two-by-two on the vertices on the sides. In the left figure we decided to park the highest cars first, in the right figure we decided to park the lowest cars first. \label{fig:non-local}}
% \end{center}
% \end{figure}
% 
% To circumvent this difficulty, we could either prove that those situations do not occur in our random setting (we did not try to do so), or
 we shall stick to a given parking procedure: \emph{park the lowest cars first}. More precisely, for each $ n \geq 0$ let us consider the finite tree $[\tau]_{n}$ made of the first $n$ generations above the root $\varnothing$ (recall that $\tau$ is supposed locally finite) together with the restriction of the car arrivals on these vertices. We can then perform the parking on $ [\tau]_{n}$ and construct variables 
 $$ x_{n}(u) \quad u \in [\tau]_{n},$$
 representing the number of cars that visited the vertices of $[\tau]_{n}$ in the parking process (recall that those variables do not depend on the order in which we parked the cars on $[\tau]_{n}$). Notice that for a given vertex $u \in \tau$, the function $ n \mapsto x_{n}( u)$ is non-decreasing (it is defined for $n$ larger than the height of $u$) so that we can let $n \to \infty$ and define 
 $$ x(u) = \lim_{n \to \infty} x_{n}(u),$$
 as the limiting number of cars visiting $u$ in the parking process on $\tau$. This morally corresponds to parking the lowest cars first\footnote{In fact, we could equivalently fix an exhaustion of $\tau$ by finite trees $\tau_1 \subset \tau_2 \subset \cdots$ and define the parking on $\tau$ as the limit of the parking procedure over the $\tau_n$'s.}. In particular we say that $u$ is void if $ x(u)= 0$, that $u$ is occupied if $ x(u) \geq 1$ and the flux of outgoing cars at $u$ is $f(u) = (x(u)-1)_{+}$.

  \subsection{Rough phase transition}
 We now focus on the case of the binary tree $ \mathbb{B}$ with i.i.d.\ car arrivals $(A_u : u \in \mathbb{B})$ with law $\mu$ satisfying $ \mu(\{0,1\})<1$. We denote by $X(u)$ the number of cars that visited vertex $u \in \mathbb{B}$ as defined in the preceding section  and will use the shorthand notation $X= X(\varnothing)$. We first establish a dichotomy on $X$ in the next lemma, which we then interpret in more geometric terms by proving that there cannot be infinite black clusters with a finite flux. %For $u \in \mathbb{B}$, we write $X_{n}(u)$ for the number of cars visiting the vertex $u$ when restricting the parking on $[ \mathbb{B}]_{n}$ as above. In particular  the  number of cars that visited the root of $ \mathbb{B}$ during the parking process is the almost sure limit $X = \lim_{n \to \infty} X_{n}(\varnothing)$.

\begin{lemma}[Dichotomy subcritical/supercritical] \label{lem:sub/super}\label{phase-transition} We have the following dichotomy:\\
 \textbf{Subcritical case.}  Either the sequence $(2^n \P(X>n) : n \geq 0)$ is bounded.\\
 \textbf{Supercritical case.} Or $X=\infty$ a.s, in which case all vertices are parked a.s.
\end{lemma}
 
% \olivier{possible de remplacer par $\E[2^\F]<\infty$ ou m\^eme  $\exists \lambda >2, \E[\lambda^\F]<\infty$}

\begin{proof}[Proof of the lemma.] Assume that $(2^n \P(X>n) : n \geq 0)$ is not bounded, and observe that the same is then true of the sequence $(2^n \P(X>n+k) : n \geq 0)$ for any integer $k$.
Then consider the collection of the $2^n$ i.i.d.~variables  $X({u})$  attached to the vertices $u$ of $\mathbb{B}$ at height $n$. We have the upper bound
$$ \P\Big(\bigcap_{u : |u| =n } \{X(u)\le n+k\}\Big) \le   \mathrm{e}^{- 2^n \P(X>n+k))},$$
with the right-hand side going to $0$ along a subsequence. But on the complement of the event on the left-hand side, one of the variables $X(u)$ is strictly larger  $n+k$ and this contribution only suffices to imply that $X= X( \varnothing) \geq k$.  Combined with our assumption, this implies that, almost surely, $X \geq k$, hence, $k$ being arbitrary, $X=\infty$. This means that the root of $ \mathbb{B}$ almost surely contains a car, and it is the same for any other vertex. \end{proof}

The next lemma says that the above dichotomy is equivalent to the existence of infinite black clusters. In particular, it rules out the possibility of having an infinite black cluster and a finite flux.
\begin{proposition} \label{prop:noblack} In the subcritical regime, there is no infinite black cluster.
\end{proposition}
\begin{proof} Suppose that $\mu$ is subcritical, so that all variables $X(u)$ are finite after the parking process. It suffices to prove that the probability that the cluster of the origin $  \mathcal{C}(\varnothing)$ is infinite is $0$. %It is easy to see that the black cluster of the origin cannot be infinite downwards, and a bit more subtle to see it cannot extend infinitely upwards.\textsc{Finiteness downwards.} Consider the parking process on $ \tilde{\mathbb{B}}$ the infinite ternary tree with a distinguished end which is obtained by grafting on the oriented line $ \mathbb{Z}$ copies of $ \mathbb{B}$ through one line. We can define the parking process on this beast using similar rule. The key is to notice that the number of cars visiting the points  of the $ \mathbb{Z}$-copie is a random walk with reflecting boundary satisfying: \begin{eqnarray*} X(k+1) & \overset{(d)}{=}& (X(k)-1)_{+} + (X-1)_{+} + A,  \end{eqnarray*} where on the right hand side the three variables are independent, and $X$ is distributed as $X( \varnothing)$ in $ \mathbb{B}$. Also, the law $X$ is stationary for this equation. It follows easily that the mean increment satisfies $ \mathbb{E}[(X-1)_{+} + A]-1 < 0$  \nico{trouver une reference de file d'attente}. In particular, $X(k)$ is positive recurrent and so the cluster of the root cannot extend infinitely downwards.\\
Fix $p \geq0$ and let us consider the event $ \mathcal{E} = \{ X( \varnothing) =p \mbox{ and }  \mathcal{C}( \varnothing) \mbox{ is infinite}\}$. We shall explore the process by parking on the first $n$ levels of $ \mathbb{B}$ as in the preceding section. More precisely, let $ \mathcal{F}_{n}$ be the sigma field generated by the variables $X_{n}(u)$ for $X_{n}(u)$ the number of cars visiting the vertex $u$ when restricting the parking on $[ \mathbb{B}]_{n}$. We then construct a sequence of stopping times $\theta_{1} < \theta_{2}  <  \cdots$ obtained as follows: $\theta_{1} = \inf\{n \geq 0 : X_{n}( \varnothing) =p \}$ and then by induction $\theta_{i+1} = \inf\{ n > \theta_{i} : \varnothing \leftrightarrow \partial [\mathbb{B}]_{n}\}$ where $\varnothing \leftrightarrow \partial [\mathbb{B}]_{n}$ means that $ \varnothing$ is connected to the level $n$ by a path whose vertices satisfy $ X_{n}(u) \geq 1$. A moment's though shows that  on the event $ \mathcal{E}$ all these stopping times are finite for otherwise the black cluster of the origin would not be infinite. For $n \geq 2$, on the event $\{\theta_n <\infty\}$, let $v_n$ be the (first, for definiteness) vertex of $\partial [\mathbb{B}]_{\theta_n}$ to be connected to the root when parking on $[\mathbb{B}]_{\theta_n}$.  Set $\mathcal{D}_n = \{  A_{v_n 0},A_{v_n 1} \in \{0,1\}\}$ the event that the two children of $v_n$ have car arrivals $\leq 1$. Plainly, $\mathcal{E} \subset  \mathcal{D}_n \cap \{\theta_n <\infty\}$ since 
otherwise, the flux coming from these two vertices would go down all the way to $ \varnothing$ and we would have $ X( \varnothing) > p$. In particular we have 
$$ \mathbb{P}( \mathcal{D}_n  \mathbf{1}_{\theta_n < \infty}\mid \mathcal{F}_{\theta_n}) = (\mu_0+ \mu_1)^{2}<1.$$  Notice then that $ \mathcal{D}_{1}, ... , \mathcal{D}_{n-1}$ are $ \mathcal{F}_{\theta_n}$-measurable so that by induction we have 
$$\P(\mathcal{E}) \leq \P(\theta_n<\infty)  \prod_{k=2}^n \P(\mathcal{D}_k |\theta_n <\infty ) \leq (\mu_0+ \mu_1)^{2(n-1)}, $$  which implies $\P(\mathcal{E})=0$ since $n$ is arbitrary and we assumed the distribution $\mu$  satisfies $\mu_0+ \mu_1<1$. \end{proof}

As a consequence of the (proof of)  Lemma \ref{lem:sub/super}, there is no lower bound for $\E[A]$ for supercritical parking, since one may cook up distributions $\mu$ with arbitrarily small expectation but $\E[2^A]=\infty$.
However, if the car arrival distribution is bounded, one can obtain  a lower bound for the expectation $\E[A]$ of the car arrival distribution for supercritical parking using a first moment method, see  \cite[Proposition 3.5]{GP19} and \cite{BBJ19} for details.% gives the following (we do not try to optimize the bound): 
%\begin{lemma}
%If there exists an integer $k\geq 2$ such that $\P(A >k)=0$, and 
%$$\E[A] < \left(4 \frac{k}{(k-1)^{(k-1)/k}}\right)^{-k},$$ then parking is subcritical.
%\end{lemma}
%
%\begin{proof} 
%Assume by contradiction that parking is supercritical, then, given $n_0$ an arbitrary integer, one may find a random integer $n \geq n_0$ and a fully parked tree (with a possible flux) with size $n$ (see the definition at Section \ref{def-fpt}) : to wit, fix the $n_0$ first vertices (in breadth-first order) of $\mathbb{B}$, and wait for the first time they are parked (in the parking procedure); at this time, the connected component of the root is a fully parked tree that contains the $n_0$ first vertices. 
%Adding the immediate neighbours of that fully parked tree in $\mathbb{B}$, we end up with a binary tree with $n+1$ leaves and there are $\frac{1}{n+1}{2n \choose n}$ (the n-th Catalan number) many such trees, which is bounded by $4^n$.
%Then must be at least $\lfloor n/k \rfloor $ vertices $u$ such that $A_{u}\geq 1$, an event with probability at most $\E[A]$.
%Now we have the bound:
%$$a_n := \frac{1}{n+1} {n \choose \lfloor n/k \rfloor} \E[A]^{n/k} \leq \left( \frac{4 k}{(k-1)^{(k-1)/k}} \E[A]^{1/k}\right)^n$$
%Choosing $\E[A]$ as in the statement as the Lemma makes the sum $\sum_{n \geq N}a_n$  converge, hence the probability that there exists an infinite component of parked vertices is less than $1$, a contradiction.
%\end{proof}

To speak of a phase transition,  one may imagine a family $(\mu^{\alpha})_{\alpha \geq 0}$ of car arrival distributions that is stochastically increasing in the mean $\alpha$. In this case, the subcritical phase is  identified with a \emph{closed set} $ \alpha \in [0,\alpha_c]$, and the supercritical phase with the set $]\alpha_c,1]$. The fact that $\alpha_{c}$ is actually subcritical (i.e.~satisfies the first alternative of the dichotomy) can be seen by monotone convergence since the expectation of the flux is bounded above by $2$ in the whole subcritical phase as recalled in the introduction (see \cite{BBJ19} for details).

\section{Decomposition into fully-parked components}
In this section we present our combinatorial decomposition which underlies our main results. The idea is very simple: we decompose the final configuration on $ \mathbb{B}$ into the black clusters of parked vertices and the white empty vertices. This shows that we can decompose the final configuration as a two-type Bienaym\'e--Galton--Watson tree whose offspring distribution is related to the generating function $F$ of fully parked trees studied in detail in the next section.
\subsection{Fully parked trees}
\label{def-fpt}
Suppose that we performed the parking process on $ \mathbb{B}$, and recall that the black vertices are those $u \in \mathbb{B}$ satisfying $ X(u) \geq 1$, the other ones being the empty or white vertices. The \emph{finite} black connected components are \emph{fully parked trees} $\tr$, i.e.\ connected subsets of the binary tree decorated by car arrivals $(a_u)_{u \in \tr}$ such that after parking all vertices are occupied. If such a tree appears as the black component of the root $ \varnothing$, then the fully parked tree may have an outgoing flux at the root (i.e.~containing more cars than vertices), otherwise it contains as many cars as vertices. See Figure \ref{fig:FPT}.

\begin{figure}[!h]
 \begin{center}
 \includegraphics[width=12cm]{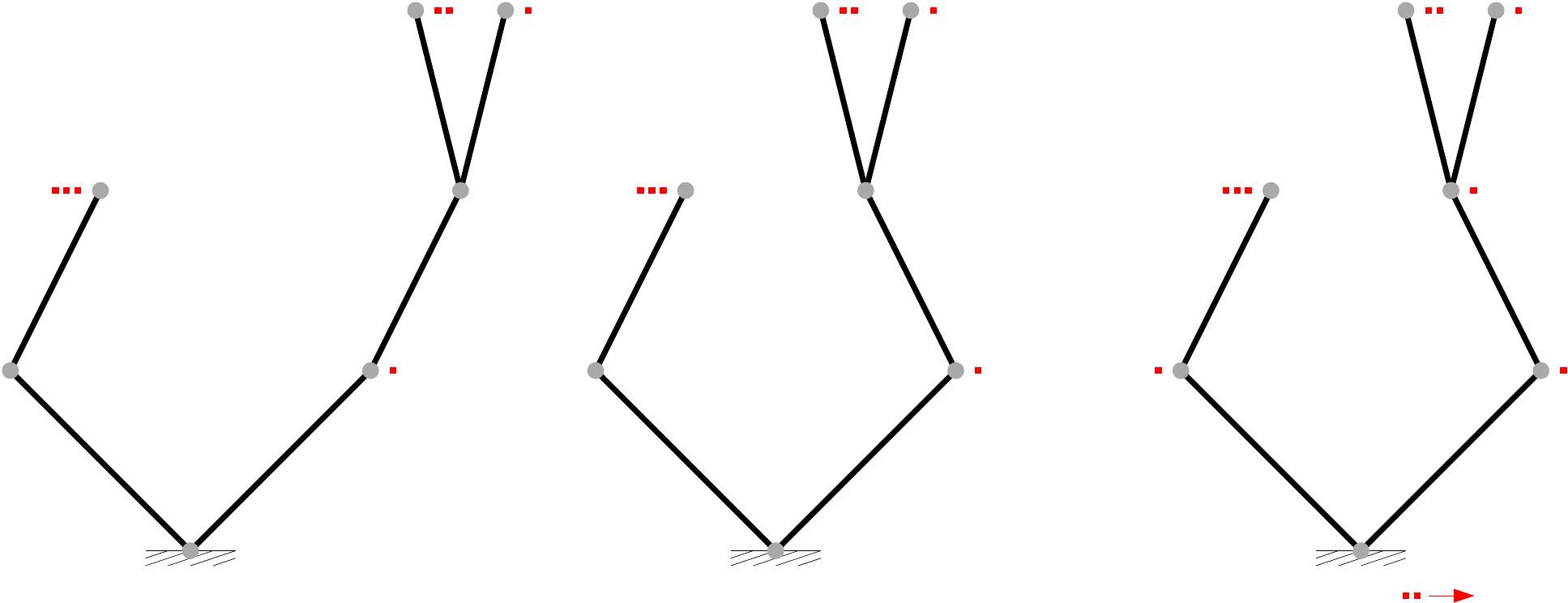}
 \caption{Three examples of fully parked trees. Notice that the first two have no outgoing flux and represent the same plane tree but their embeddings in $ \mathbb{B}$ is different. The last fully parked tree on the right has an outgoing flux of $2$ cars. \label{fig:FPT}}
 \end{center}
 \end{figure}
 
 For the enumeration of the fully parked trees we shall always consider that their bottom vertex is $\varnothing$. Each plane rooted structure of a fully parked tree with $m$ vertices with $1$ child actually corresponds to $2^m$ different embeddings as a subset of $  \mathbb{B}$ (with $ \varnothing$ as the root): for this reason, later in the decomposition we shall put a weight of $2$ for vertices with outdegree $1$. Let us denote by $ \mathbb{T}_n^{(p)}$ the set of all fully parked trees with root $\varnothing$, with $n$ vertices and having outgoing flux $p \geq 0$ (i.e.~$p+1$ cars have visited the root vertex). The weight $w({\tr})$ of a fully parked tree $ \tr \equiv (\tr : (a_u : u \in \tr))$ is the weight of its car decoration, that is 
 $$w({\tr})= \prod_{u \in \tr} \mu_{a_u}.$$
 We can then form the bivariate generating series of fully parked trees (with flux) as 
 \begin{eqnarray*} F(x,y)  \equiv F_{\mu}(x,y) &:=& \sum_{n \geq 1}\sum_{p \geq 0} \sum_{\tr \in \mathbb{T}_n^{(p)}} x^n y^p w({\tr})\\ &=&  x  \big(\mu_1+ \mu_2 y + \mu_3 y^2 + \cdots\big) + x^2 \big(2(\mu_2+ \mu_1^2) +  2y(\mu_3+ 2\mu_1\mu_2) + \cdots \big)  + \cdots   \end{eqnarray*}
 Section \ref{sec:enumeration} is devoted to the study of $F$ via a functional equation obtained by splitting a fully parked tree at the root, see \eqref{eq:tuttesimple}. But before doing so, let us present the combinatorial decomposition and the characterization of subcriticality in terms of $F$. It turns out that most equations simplify if one introduces
 $$ \mathbf{F}(x,y) := 1 + F(x,y) \quad \mbox{ and }\quad  \mathbf{F}_{0}(x) := \mathbf{F}(x,0) := 1+ F(x,0).$$

 \subsection{Decomposition} 
Recall that the law of $X$ is $(p_k : k \geq 0)$ and that we gave a short-hand notation $p_\circ =p_0$ for the probability that the root vertex is empty. We write $ \mathcal{C}( \varnothing)$ for the monochromatic cluster of the origin in $ \mathbb{B}$ after parking. Notice that the number of vertices adjacent from above to a fully parked tree $\tr \subset \mathbb{B}$ with $n$ vertices is $n+1$, regardless of the shape of $\tr$. Recalling Proposition \ref{prop:noblack} we have for $k \geq 0$
 \begin{eqnarray}  \mathbb{P}( X( \varnothing) = k+1) &\underset{ \mathrm{Prop.} \ref{prop:noblack}}{=} &\mathbb{P}( X( \varnothing) = k+1 \mbox{ and } \# \mathcal{C}( \varnothing) < \infty) \nonumber \\& =& \sum_{\tr \in \mathbb{T}_n^{(k)}} \mathbb{P}(\mathcal{C}( \varnothing) = \tr)\nonumber \\&=&  \sum_{n \geq 1} \sum_{\tr \in \mathbb{T}_n^{(k)}} w( \tr) p_\circ^{n+1}   = p_\circ [y^k]F(p_\circ,y). \label{eq:funda1}\end{eqnarray}
The other fundamental equation is obtained by noticing that the event $\{X(\varnothing)=0\}$ occurs if and only if $A_\varnothing=0$ and $\{X(u) \in \{0,1\} : \mbox{ for } u \in \{0,1\}\}$ which turns into 
 \begin{eqnarray} p_\circ &=& \mu_0 ( p_\circ+ p_\bullet)^2.  \label{eq:funda2}\end{eqnarray}
Specializing  \eqref{eq:funda1} to $k=0$, we recover together with the previous display the fixed point equation $p_\bullet = p_\circ {F} (p_\circ,0)$, mentioned in the introduction.  In particular, re-injecting in  \eqref{eq:funda2} we obtain
  \begin{eqnarray*} \quad p_{\circ} = \mu_{0} p_{\circ}^{2}\big(  \mathbf{F}_{0}( p_{\circ})\big)^{2}, \quad \mbox{ where we recall that } \mathbf{F}_{0}(x) = 1+F(x,0).   \label{eq:funda3}
  \end{eqnarray*}
Notice that the function $ x \mapsto \mu_{0} x^{2} (\mathbf{F}_{0}(x))^{2}$ is strictly convex and that $p_{\circ}=0$ is a trivial solution to the above equation, so there is \textbf{at most one positive solution} $p_{\circ}$. Under the same hypothesis, splitting according to the values of $X(\varnothing)$ we also obtain thanks to \eqref{eq:funda1}
 \begin{eqnarray} 1 &=& p_{\circ} + p_{1}+ p_{2}+ \cdots \underset{\eqref{eq:funda1}}{=} p_{\circ} \mathbf{F}(p_{\circ},1). \label{eq:funda4}  \end{eqnarray}

 \begin{proposition}[$F$-characterization of subcriticality] \label{prop:criti} \label{prop:decomp}
 The law $\mu$ is subcritical if and only if there is a positive solution to the 
 equation  \begin{eqnarray} \quad 1 = \mu_{0} x(  \mathbf{F}_{0}(x))^{2}. \label{eq:fundabis}
 \end{eqnarray}\end{proposition}
% two equations  \begin{eqnarray} \quad 1 = \mu_{0} x(  \mathbf{F}_{0}(x))^{2} \quad \mbox{ and } \quad  1=x \mathbf{F}(x,1).  \label{eq:fundabis}

\begin{proof} Let $\mu$ be a subcritical law for the parking on $ \mathbb{B}$. Since $p_{\circ} \ne 0$, the above calculations show that $p_{\circ}$ is indeed a solution to the equation \eqref{eq:fundabis}. 

Conversely, suppose that there is a positive solution $x_{\circ}$ to  \eqref{eq:fundabis}. As a special case of equation \eqref{eq:tuttesimple} below for $ \mathbf{F}(x,y)$, we know that the series $  \mathfrak{f}(y) = \mathbf{F}(x_{\circ},y)$ is a solution to $ y +  \mathfrak{f}^2 x_{\circ} G(y) -  y \mathfrak{f}- 1 =0$.  Solving the quadratic equation and taking the combinatorial solution we have 
$$ \mathfrak{f}(y) = \frac{y + \sqrt{y^{2}+ 4x_{\circ}(1-y)G(y)}}{2x_{\circ} G(y)}.$$
At first, the above equality holds only as a formal power series in $y$. But notice that the function $y^{2}+ 4x_{\circ}(1-y)G(y)$ inside the square-root does not vanish over $y \in [0,1]$ so that the solution above is analytic over $[0,1]$. By Pringsheim's theorem  \cite[Theorem IV.6 p.240]{Flajolet:analytic}, the function $ \mathfrak{f}$ has radius of convergence at least $1$ and we have $ \mathfrak{f}(1) = \frac{1}{x_{\circ}}$ which is   $x_\circ \mathbf{F}(x_\circ,1)=1.$
%\olivier {on peut aussi faire sans resoudre, voir ci-dessous, mais pour garantir la finitude de $\mathbf{F}(x_\circ,1)$ (non trivial a priori) il faut de toute facon la sol...\\
%
%First notice that Tutte's equation \eqref{eq:tutteintro} evaluated at $(x_{\circ},1)$ gives :
%$$x_\circ F(x_\circ,1)) =   \big( x_\circ \big(1+F(x_\circ,1)\big)^2   - x_\circ  \cdot \mu_{0} x_\circ (1+F(x_\circ,0) )^2  .$$
%which using \eqref{eq:fundabis} and rearranging gives
%$x_{\circ} \mathbf{F}(x_\circ,1)=1$.}
This in turn ensures that there exists a random variable $Z$ (the outgoing flux of cars) whose generating function is 
$$  \mathbf{Z}(y) =  x_{\circ}  \mathbf{F}(x_{\circ},y).$$ 
%the fact that the law of $Z$ is well defined follows from the second equation of \eqref{eq:fundabis}. 
We then compute, using Tutte's equation \eqref{eq:tutteintro} (see \eqref{eq:tuttesimple} below) as well as  \eqref{eq:fundabis}:
\begin{align*}
& \frac{1}{y} \big( \mathbf{Z}(y)^{2} G(y) -  \mathbf{Z}(0)^{2} G(0) \big) + \mathbf{Z}(0)^{2} G(0)  \\
=&   x_{\circ} \left(\frac{x_{\circ}}{y} \big(   \mathbf{F}(x_{\circ},y) ^2 G(y) -  \mathbf{F}(x_{\circ},0)^2 G(0) \big) \right)+  x_{\circ}^2  \mathbf{F}(x_{\circ},0)^2 G(0)\\
=& x_\circ F(x_\circ,y) +  x_{\circ}  
 = \mathbf{Z}(y),
 \end{align*}
%Using the forthcoming Tutte's equation \eqref{eq:tuttesimple} on $ \mathbf{F}$ and the first equation on $x_{\circ}$, we check \nico{a faire!!!} 
%that we have 
% \begin{eqnarray} 
%  \mathbf{Z}(y) = \frac{1}{y} \big( \mathbf{Z}(y)^{2} G(y) -  \mathbf{Z}(0)^{2} G(0) \big) + \mathbf{Z}(0)^{2} G(0)  
%  \label{stat-Z}
% \end{eqnarray} 
but this identity is equivalent to the  following recursive distributional equation for $Z$:
$$Z \overset{(d)}{=} (Z_{1}+Z_{2} +A-1)_{+},$$ where on the right-hand side the variables are independent and $Z_{1},Z_{2}$ are two copies of law $Z$.

%Using reverse engineering, 
This recursive distributional equation enables us to decorate the vertices of $ \mathbb{B}$ by i.i.d.\ variables $A_{u}$ in such a way that for every $n$, the parking on $[ \mathbb{B}]_{n}$ together with i.i.d.\ fluxes on $\partial [ \mathbb{B}]_{n}$ yields a flux of law $Z$ at the root (in a coherent manner). Replacing  the i.i.d.\ fluxes on $\partial [ \mathbb{B}]_{n}$ by null fluxes on $\partial [ \mathbb{B}]_{n}$, and writing $X_n = X([ \mathbb{B}]_{n})$
for the number of cars visiting the root for the parking on $[ \mathbb{B}]_{n}$, 
we deduce by comparison that the flux at the root $\mathrm{F_n} = (X_n-1)_{+}$ is dominated by the a.s. finite random variable $Z$, which implies in particular that the parking on $ \mathbb{B}$ with law $\mu$ is subcritical. Et voil\`a.
\end{proof}
%\alice{On a perdu l'equation sur $p_{ \bullet}$, non? }
%\begin{remark} Law of the two-type tree. \end{remark}

%-> equations sur p0, pBullet, pk.
\section{Enumeration of fully parked trees} \label{sec:enumeration}
This section is the analytic core of the paper. We write the recursive equations (Tutte's equation) for fully parked trees and solve them using the kernel method of Bousquet-M\'elou \& Jehanne \cite{BMJ06}. Combined with Proposition \ref{prop:decomp} this enables us to prove our main results easily. The results are similar to the work of Chen \cite{chen2021enumeration} which considered plane fully-parked trees (as opposed to our binary case). Notice also that the technical part of \cite{chen2021enumeration} consists in obtaining asymptotics for the coefficients, a goal that we did not pursue in these pages.
\subsection{Solving Tutte's equation}

Recall that $F(x,y)$ is the bivariate generating function of the fully parked trees where $x$ encodes the number of vertices of the tree and $y$ the \textbf{flux} of cars and $G$ is the generating function of the car arrivals. To enumerate  fully parked trees, we decompose them at their root vertices. Take a fully parked tree with $n \geq 1$ vertices and $n+p$ cars in total (the flux of cars is $p$). Then
\begin{itemize}
\item either $n=1$ which means that the root vertex has no vertex above it. In this case, at least one car arrives on this vertex (since the root vertex should contain a parked car) and the number of cars arriving on this vertex is $1+p$. Summing over $p$  gives the term $ x (G(y)- G(0))/y$.
\item Another possibility is that the root vertex has a unique child in the fully parked tree, which can be the left or right neighbor in $ \mathbb{B}$. In that case, the subtree above this child is a fully parked tree with $n-1$ vertices and a flux of cars $p_1$ where $p_1+ \ell -1 =p$ if there are $ \ell$ cars arriving on the initial root vertex. Notice that the case $p_1 = \ell =0$ is excluded since otherwise the root vertex is not parked. Summing over $p$ yields the term $ 2x(F(x,y)G(y)- F(x,0)G(0))/y$. 
\item The last case is when the root vertex has two parked children, each carrying a fully parked tree above it with respective sizes $k \geq 1$ and $n-k-1 \geq 1$ and flux of cars $p_1$ and $p_2$. To obtain a flux of cars $p$ at the root, one must have $p_1+p_2+\ell - 1 = p$ where $ \ell$ is the number of cars arriving at the root vertex. Again the case $p_1=p_2= \ell=0$ is excluded.  We thus obtain a term $x (F(x,y)^2 G(y) - F(x,0)^2G(0))/y.$
\end{itemize}

\begin{figure}[!h]
 \begin{center}
 \includegraphics[width=15cm]{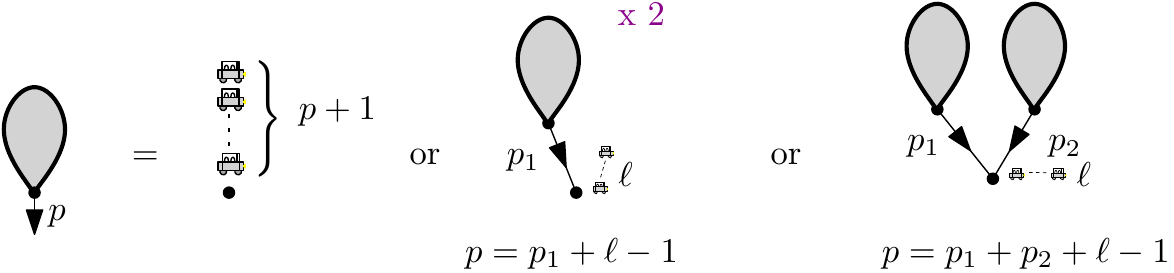}
 \caption{\label{fig:tutte} Illustration of Tutte's recursive decomposition at the root vertex. On the left a fully parked tree with flux $p$. If $n=1$, then the tree is just a vertex with $p+1$ cars arriving on it. Otherwise, it has one or two children which are the root of smaller fully parked tree. }
 \end{center}
 \end{figure}
Summing these three terms, we obtain the following recursive equation for $F$:
\begin{equation} \label{eq:Tutte} yF(x,y) = x(G(y)-G(0)) + 2x(F(x,y)G(y)- F(x,0)G(0))+ x (F(x,y)^2 G(y) - F(x,0)^2G(0))\end{equation}
 With our notation $ \mathbf{F} = F+1$ and $ \mathbf{F}_0 (x) = \mathbf{F}(x, 0) = F(x,0)+1$, this equation simplifies to 
  \begin{eqnarray} P( \mathbf{F}(x,y), \mathbf{F}_0(x),x,y)=0 \quad \mbox{ where } \quad  P(f,f_0,x,y) = y + f^2 x G(y) - f y - f_0^2 x G(0).  \label{eq:tuttesimple}  \end{eqnarray}

To solve this equation, we apply the kernel method of Bousquet-M\'elou and Jehanne \cite{BMJ06} and look for a (formal) power series $ Y=Y(x)$ such that $\partial_{f} P( \mathbf{F}(x, Y(x)), \mathbf{F}_0(x),x,Y(x))=0$ so that combined with \eqref{eq:tuttesimple} we also find automatically $\partial_{y} P( \mathbf{F}(x, Y(x)), \mathbf{F}_0(x),x,Y(x))=0$. %``m\'ethode de Mireille" even though $P$ is not a polynom: consider the system of equations
%\begin{equation}  \label{eq:system1}
%\left\lbrace 
%\begin{array}{l}
%\partial_{f} P(f,f_0, x,y) =0 \\
%\partial_{t} P(f,f_0, x,y) =0 \\
% P(F,f_0, x,y) =0.
%\end{array}\right.
%\end{equation}
%\nico{Pour guider le lecteur, dire que l'on cherche $Y=Y(x)$ pour avoir je-ne-sais-plus-quelle equation, puis la troisieme s'en deduit. Ne pas melanger les variables formelles pour la definition du polynome et les vrais variables $F,G$. Utiliser plutot $f,g$. Dans la suite garder $Y=Y(x)$ non?}
This introduction may seem ad-hoc, but it enables us to find a system of three equations on the three unknowns $ \mathbf{F}, \mathbf{F}_{0}$ and $Y$, so that with a little luck we will find an ``expression'' for those. Actually, as we will see below $x \mapsto Y(x)$ is a convenient change of variable which simplifies our expressions. To summarize, we are looking for a solution $Y \equiv Y(x)$ to the following system:% and deduce from it an ``explicit" formula for $ \mathbf{F}_0$ anf $ \mathbf{F}$. 
\begin{equation}  \label{eq:system}
\left\lbrace 
\begin{array}{l}
Y - 2 x \mathbf{F} G(Y)=0, \\
1 +  x  G'(Y) \mathbf{F}^2= \mathbf{F}, \\ 
Y + xG(Y) \mathbf{F}^2= Y \mathbf{F}+xG(0) \mathbf{F}_0^2 .\end{array}\right.
\end{equation}
Thanks to the first equation, we know that $ \mathbf{F}= Y/(2xG(Y)).$ Replacing $ \mathbf{F}$ by this quantity in the second equation, we obtain 
 \begin{eqnarray} 1 + \frac{Y^2 G'(Y)}{4xG(Y)^2} = \frac{Y}{2xG(Y)} \quad \mbox{ that is } \quad  Y = x \left( \frac{4G(Y)^2}{2G(Y) - YG'(Y)}\right),   \label{eq:lagrange}\end{eqnarray} which makes it clear that $Y \equiv Y(x)$ exists as a power series (and even with a positive radius of convergence in a neighborhood of $0$). Once the existence of $Y$ is granted, we use again the system of equations \eqref{eq:system} to obtain an equation that only involves $\mathbf{F}_0(x)$ and $Y (x). $
If we replace in the third equation $ \mathbf{F}$ by $Y/(2xG(Y)) $ (which is a consequence of the first equation) and $x$  by $\frac{Y (2G(Y) - YG'(Y))}{4G(Y)^2}$, we obtain 
$$ \frac{4YG(Y)}{2G(Y) -Y G'(Y)}+ \frac{ \mathbf{F}_0^2 Y G(0) (2G(Y) -Y G'(Y))}{G(Y)^2}= 4 Y.$$
This equation is quadratic in $\mathbf{F}_0(x)$ and using the fact that it has non negative coefficients we obtain 

\begin{equation} \label{eq:F1}\mathbf{F}_0(x) =   \frac{2 G( Y) \sqrt{G( Y) - Y G'( Y)}}{  (2 G( Y) - YG'( Y))\sqrt{G(0) }} \quad \mbox{with} \quad Y\equiv Y(x) \mbox{ as in \eqref{eq:lagrange}}.
\end{equation}
We have found here an ``explicit" solution for $\mathbf{F}_0(x)$ around $x =0$.
Coming back to Tutte's equation, once $ \mathbf{F}_{0}$ is known this equation is quadratic in $ \mathbf{F}$ and we can solve it into
\begin{eqnarray}
\label{Fxy}
\mathbf{F}(x,y) = \frac{y \pm \sqrt{y^2+ 4 G(y) x (G(0)  \mathbf{F}_0(x)^2 x - y) }}{2 G(y) x}.
\end{eqnarray}
The sign in front of the square root can actually change since the function inside the square root vanishes when  $y = Y(x)$ and we need to change branch to keep an analytic function. But we shall not use the exact expression in what follows.
%\nico{Actually the square-root can vanish (it should be factored by $(y-Y(x))^{2}$) so defining the sign is a slight issue. We need to think about the best way to present it}
%\nico{Eh beh!!!} \alice{What? C'est juste Tutte qui est quadratique ca? Je suis un peu inquiete du x au d\'enominateur.} \nico{Oui, un peu bizarre} \nico{C'est verifie?} \alice{Mmm... Faut trouver quoi Ã©crire}

\subsection{Radius of convergence} \label{sec:Y(x)}
In this section we use the explicit resolution of the functional equation \eqref{eq:tuttesimple} to  determine the radius of convergence $x_{c}$ of $ \mathbf{F}_{0}$ and the value of $ \mathbf{F}_{0}(x_{c})$. The important fact for our application to parking being that under the condition $(\star)$ on the existence of $t_c$ in Theorem \ref{thm:phase} we have 
$$  \mathbf{F}_0(x_c) =  \frac{2 G( t_c) \sqrt{G( t_c) - t_c G'( t_c)}}{  (2 G( t_c) - t_cG'(t_c))\sqrt{G(0) }} .$$ 
Recall from \eqref{eq:F1} that $ \mathbf{F}_0(x)$ is an explicit function of $Y(x)$ itself given by the implicit equation \eqref{eq:lagrange}. 

\paragraph{Analyticity of $Y$.} We first determine the analytic properties of the change of variable $x \mapsto Y(x)$. Recall from \eqref{eq:lagrange} that $x$ and $Y=Y(x)$ are linked by the equation 

 %so that $$ x = \frac{}{}
%Take the second equation and replace in the second term the factor $G(y)$ by $2x(F+y)G(y)^2/y^2$ (using the first equation) and multiply the whole equation by $y^2 G(y)$ so that the second term is divisible by $ (F+y)^2$. We obtain $$2 y^3 (F + (1 + F_1)^2 x G(0))G (y) = xG(y) (F + y)^2 (4 x F(y)^2 + y^2 G'(y))).$$
%Now on the RHS, we replace the factor $x(F+y)^2 G(y) $ by $y^2(F+ (1+F_1)^2 x G(0)$ and obtain 
%$$(T + (1 + T1)^2 x G(0)) (-2 y G(y) + 4 x G(y)^2 +  y^2 G'(y))=0.$$
%Since the $T$ in the first factor is a function of $y$, we should have
\begin{equation} \label{eq:xfcty} x= \frac{Y (2G(Y) - YG'(Y))}{4G(Y)^2}, \quad \mbox{ equivalently}\quad  x = \psi(Y(x)) \quad \mbox{ with } \psi(y) =\frac{y (2G(y) - yG'(y))}{4G(y)^2}.  \end{equation}% \qquad Q(x,Y) = 0 . \end{equation} where $$ Q(x,y) = y^2 G'(y) - G(y)(2y-4xG(y)).$$ Along the curve defined by $Q(x,y)= 0$, we have  \begin{eqnarray*} \partial_{y}Q(x,y) &=& (8xG'(y)- 2)G(y) + y^2G''(y) \\ &=& \frac{y^2G''(y)G(y)-2(G(y)-yG'(y))^2 }{G(y)}. \end{eqnarray*}
Note that 
$$ \partial_{y} \psi(y) = \frac{2(G(y)-yG'(y))^{2}-y^{2}G(y)G''(y)}{4G(y)^{3}}$$ and in particular $\psi(0)=0$ and $\psi'(0) >0$ so that by the implicit function theorem, we can define $ Y$ in a neighborhood of $0$ such that $ x = \psi(Y(x))$. Recall the condition $(\star)$ from Theorem \ref{thm:phase} which says that the function $ y \mapsto y^2G''(y)G(y)-2(G(y)-yG'(y))^2$ at the numerator of $\partial_{y} \psi(y)$ reaches $0$ at time $t_c \in (0,\infty)$, see Figure	 \ref{fig:scenarii}.
\begin{remark}  \label{rek:RCVinfty} In particular if $G$ has an infinite radius of convergence then $(\star)$ holds. Indeed, the quantity $ G(y) - y G'(y) = \sum_{k \geq 0} \mu_k (1-k) y^k$ equals $G(0) = \mu_0 >0$ at $y=0$ and is bounded from above by $ \mu_0 - (1- \mu_0 - \mu_1) y^2$ for $y >1$ which goes to $- \infty$ as $ y \to + \infty$. Thus, there exists $z_c$ such that $ G(z_c) - y_c G'(z_c) = 0$ and the function $ y \mapsto y^2G''(y)G(y)-2(G(y)-yG'(y))^2$ is positive at $y = z_c$. Since it is negative at $y=0$, then ($\star$) holds and $t_c \in (0,z_c).$ \\ 
%The assumption $(\star)$ is also satisfied when $G$ has a finite radius of convergence $y_c$ and $\lim_{y \to y_c} G'(y) = \infty$  since we can apply the same argument as above  or when $\lim_{y \to y_c} G''(y) = \infty$ since we would also have $\lim_{y \to y_c} y^2G''(y)G(y)-2(G(y)-yG'(y))^2 = + \infty.$ 
The assumption $(\star)$ is also satisfied when $G$ has a finite radius of convergence $y_c$
and (at least) one of the three quantities $G(y_c)$, $G'(y_c)$, $G''(y_c)$ is infinite.
In case $G(y_c)=\infty$, starting from $y G'(y)-G(y) = y  \sum_{k \geq 0} \mu_k (k-1) y^{k-1}$, and noting that $\mu_k (k-1) y^{k-1}\sim \mu_k k y^{k-1}$, we deduce that $y G'(y)-G(y) \sim y G'(y)$ as $y \to y_c$, hence $G(y) - y G'(y) $ again has limit $-\infty$ as $y \to \infty$, and the same argument as above applies.
The last cases are obvious : in case $G'(y_c)=\infty$ but $G(y_c)<\infty$, $G(y) - y G'(y) $ plainly has limit $-\infty$; last, in case $G''(y_c)=\infty$ but $G'(y_c)<\infty$, we directly have $\lim_{y \to y_c} y^2G''(y)G(y)-2(G(y)-yG'(y))^2 = + \infty.$ 
% \nico{Et meme si G' ou G'' pete au RCV....}
\end{remark}
To clarify the reader's spirit and for latter discussion, let us classify the possible scenarios according to the three cases identified in  Chen \cite[Figure 1]{chen2021enumeration}, see Figure \ref{fig:scenarii}:
\begin{itemize}
\item the most common case is when $t_{c}$ exists and is strictly less than the radius of convergence $y_{c}$ of $G$. At this point we have 
$$ \left. \frac{ \partial ^{2}}{\partial y \partial y}\psi(y)\right|_{y=t_{c}} =  \frac{t_{c}}{4G(t_{c})^{3}}  \left( 2(t_{c}G'(t_{c})-2G(t_{c}))G''(t_{c})-t_{c}G(t_{c})G'''(t_{c})\right),$$ and since 
$G(t_{c}) \geq t_{c}G'(t_{c})$ this second derivative is strictly negative so that the function $ y \mapsto \psi(y)$ reaches a local maxima at this point. We call this situation the \textbf{generic} situation.
\item it could also happen that $t_{c}$ exists and is equal to the radius of convergence of $G$. In this case, although $y \mapsto \psi(y)$ reaches a maxima at $t_{c}$, the local behavior around the maximum may not be quadratic. We call this situation the \textbf{non-generic situation}.
\item Otherwise $t_{c}$ does not exists and in particular the radius $y_{c}$ of convergence of $G$ is finite and $y \mapsto \psi(y)$ has a finite positive derivative at $y_{c}$. This is the \textbf{dense} situation which leave aside for the moment.
\end{itemize}

\begin{figure}[!h]
 \begin{center}
 \includegraphics[width=15cm]{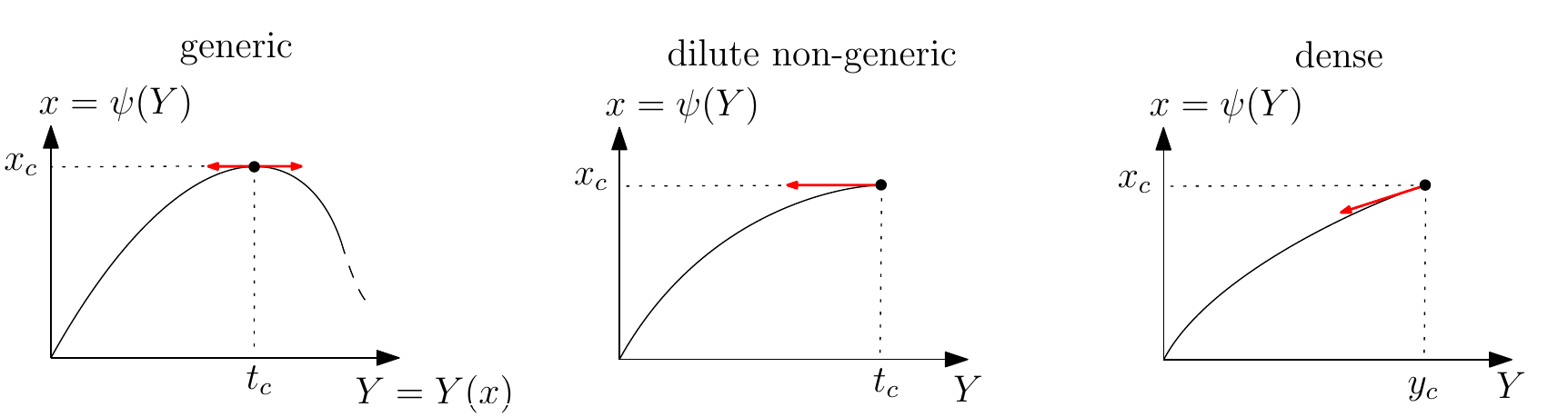}
 \caption{Illustration of the three scenarios  for the change of variable $ x \mapsto Y(x)$. Our standing assumption $(\star)$ holds in the first two cases. \label{fig:scenarii}}
 \end{center}
 \end{figure}

Under the assumption $(\star)$ --i.e.\ except in the dense situation-- we introduce 
 \begin{eqnarray} \label{eq:defxc} x_{c} = \psi(t_{c})=\frac{t_{c} (2G(t_{c}) - t_{c}G'(t_{c}))}{4G(t_{c})^2}. \end{eqnarray}
\begin{lemma} Under assumption $( \star)$ the function $x \mapsto Y(x)$ is increasing and  analytic over $[0,x_c)$ and furthermore 
$$ \lim_{x \to x_c} Y'(x) = \infty.$$
\end{lemma}
\begin{proof} Under the assumption $(\star)$ the function $y \mapsto \psi(y)$ is increasing and analytic over $[0,t_{c})$ so that by the analytic version of the implicit function theorem one can define its increasing  inverse function $ x \mapsto Y(x)$ over $[0,x_{c})$. Note that since $\psi'(y) \to 0$ as $y \uparrow y_{c}$ we have $Y'(x) \to \infty$ as $x \uparrow x_{c}$.%so that are linked by $x= \frac{Y (2G(Y) - YG'(Y))}{4G(Y)^2}.$ Differentiating this expression with respect to $x$, we obtain 
%\begin{equation} \label{eq:yderivee}
%1 = \frac{Y'(x) (2 (G(Y(x)) - Y(x) G'(Y(x)))^2 - Y(x)^2 G(Y(x))  G''(Y(x))}{4 G(Y(x))^3}.\end{equation}
%Thus, using the fact that under assumption $( \star)$, the function $x \mapsto Y(x)$ converges towards  $t_c$ as $x \to x_c$, we obtain
%\begin{eqnarray*} \lim_{x \to x_c} Y'(x) %=& \lim_{x \to x_c} \frac{4 G(Y(x))^3}{ 2 (G(Y(x)) - Y(x) G'(Y(x)))^2 - Y(x)^2G(Y(x))  G''(Y(x)}\\ 
%= \lim_{t \to t_c} \frac{4 G(t)^3}{ 2 (G(t) - t G'(t))^2 -  t^2 G(t)G''(t)} = + \infty
%\end{eqnarray*} 
 \end{proof}

\paragraph{Radius of convergence of $ \mathbf{F}_0$.} We still suppose $ (\star)$. Coming back to $ \mathbf{F}_0$, notice  that $ (G(Y(x))-Y(x)G'(Y(x)))$ is always positive for $x \in [0,x_c]$,  thus by Equation \eqref{eq:F1} the function $ \mathbf{F}_0$ is also analytic over $[0,x_c)$. Since $ \mathbf{F}_0$ only has positive coefficients, by  Pringsheim's Theorem (see for example \cite[Theorem IV.6 p.240]{Flajolet:analytic}), its radius of convergence is \emph{at least} $x_c$. The following lemma shows that it actually coincides with it 
\begin{lemma} Suppose $( \star)$ then we have 
$$ \lim_{x \to x_c} \mathbf{F}''_0(x)= \infty,$$
in particular the radius of convergence of $ \mathbf{F}_0$ is equal to $x_c$.
\end{lemma}
\begin{proof}%\alice{Je ne sais pas si on en a besoin mais je le donne quand mÃªme } 
We use our explicit computations of $ \mathbf{F}_0$ and $ Y$ to derive formulas for the first two derivatives of $x$. Even if we don't need it for this lemma, we start with the first derivative. We take the expression of $ \mathbf{F}_0$ given by Equation \eqref{eq:F1}, differentiate it with respect to $x$ and replace the occurence of $ Y'(x)$ by $1/ (\partial_{y} \psi (Y(x)))$ thanks to Equation \eqref{eq:xfcty}. We obtain 
$$ \mathbf{F}'_{0} (x) = \frac{4 G(Y(x))^3 G'(Y(x))}{\sqrt{G(Y(x))-Y(x) G'(Y(x))} (2 G(Y(x)) -Y(x) G'(Y(x)))^2}.$$
This quantity has a finite limit when $ Y(x)$ converges to $t_c$. We thus need to compute the second derivative of $ \mathbf{F}_0$. To do so, we differentiate the above expression of $ \mathbf{F}'_0$ and again replace the occurence of $ Y'(x)$ by $1/ (\partial_{y} \psi (Y(x)))$. We then obtain a fraction involving the derivatives of $G$ at $Y(x)$. %\alice{Je crois que c'est pas la peine de l'ecrire celle la...} 
Using the definition of $t_c$ under assumption $(\star)$, we can show that 
\begin{eqnarray*} \lim_{x \to x_c} \mathbf{F}_0'' (x) = \lim_{t \to t_c}\frac{16 G(t)^6 (2 G(t) - t G'(t))^3 (G(t) -  t G'(t)) (2 G(t) - t G'(t))^2}{t^2 (G(t) - t G'(t))^{3/2} \cdot (2 (G(t) - t G'(t))^2 -  t^2 G(t) G''(t))} = + \infty, 
  \end{eqnarray*}
since all but the factor $(2 (G(t) - t G'(t))^2 -  t^2 G(t) G''(t)$ are positive for $t< t_c$ and have a positive limit as $t \to t_c$. 
\end{proof}

%\nico{donner les valeurs qu'on utilise apres...}

%En ÃÂÃÂ©liminant $F$ et $x$, on obtient $$(1+F_1(y(x))^2 = \frac{4 G(y)^2 (G(y)-yG'(y))}{G(0)(2G(y)-yG'(y))^2}$$

%Using $(1+F_1(p_{\circ}))^2 = 1/(G(0) p_{\circ})$, we obtain $((y-2)G(y)-y(y-1)G'(y))(yG'(y) - 2G(y))=0$ where $y = y(p_{\circ})$.
%Actually, the smallest positive solution $y$ is the solution of $(y-2)G(y)-y(y-1)G'(y) = 0$

\section{Probabilistic consequences}

\label{sec:proba}
Armed with our enumeration results and the  criterion  of Proposition \ref{prop:criti}, we can now proceed to the proof of our main results.
\subsection{Theorem \ref{thm:phase}: Location of the threshold}
%manipulate -> 1 solution pour p0 est dans ce cas succinique.
%soit 2 solutions, la plus grande est physique. Classification critique ou sous-critique (est-ce que \c ca pÃÂÃÂ\c cÃÂÃÂ\`ete ÃÂÃÂ\c cÃÂÃÂ\`a la deuxiÃÂÃÂ\c cÃÂÃÂ\`eme solution).

Recall that by Proposition \ref{prop:criti}, the parking process is subcritical if and only if there exists a positive solution to \eqref{eq:fundabis}. When $(\star)$ holds, since the function $ x \mapsto G(0) x \mathbf{F}_{0}(x)$ is strictly increasing, Equation \eqref{eq:fundabis} has a solution if and only if $G(0) x_c \mathbf{F}_{0}(x_c)^2 \geq1 $ where $x_c$ is the radius of convergence of $ \mathbf{F}_0$ found in the previous paragraph.  Now, since $t_c = Y (x_c)$ and plugging the value of $\mathbf{F}_0(x_c)$ given by Equation \eqref{eq:F1} in $G(0) x_c \mathbf{F}_{0}(x_c)^2 \geq1$, we obtain $t_c+ \frac{t_c G(t_c)}{t_c G'(t_c) - 2 G(t_c)} \geq 1$.

By definition of $t_c$ under Assumption ($\star$), the quantity $t_cG'(t_c) - 2G(t_c)<t_cG'(t_c) - G(t_c)$ is always negative. Hence there is a solution $x$ to \eqref{eq:fundabis} if and only if  $t_c$ satisfies $$(t_c-2)G(t_c) \geq t_c(t_c-1)G'(t_c), $$
which together with Proposition \ref{prop:criti} concludes the proof of Theorem \ref{thm:phase}. 
%And we recall that since $t = Y(p)$ for some $p \leq x_c$, we also need
%$t^2G''(t)G(t)-2(G(t)-tG'(t))^2 \leq 0.$
%\alice{C'est la ou que tu m'as dit que c'etait necessaire parce que les equations sont gentilles? ou faut verifier qu'on peut pas passer la singularite?}

\begin{remark}\label{tcgeq2} Notice that if $t_c$ satisfies the condition of Theorem \ref{thm:phase}, then it is greater than $2$. This implies that if the parking process is subcritical, then the radius of convergence of the generating function $G$ of the car arrivals is at least $2$. To see it, first note that the inequality $(t-2)G(t) \geq t(t-1)G'(t)$ can not be satisfied for $t \in [1,2]$ since the left-hand side is non-positive and equals $0$ only for $t=2$, whereas the right-hand side is non-negative and equals $0$ for $t=0$ and $t=1$ only. 
Neither can this inequality be satisfied for $ t \in (0,1)$, because for such $t$, we can bound from above the quantity $  t(t-1)G'(t)/(t-2)$ by $ (3 -2\sqrt{2})G'(1) \leq  (3 -2\sqrt{2})G'(1) \leq 1/5$ since $ G'(1) = \mathbb{E}[A] \leq 1$ in the subcritical case. On the other side, the quantity  $G(t)$ is bounded from below by $G(0) \geq 1/2.$
\end{remark}

\subsection{Theorem \ref{thm:critical}: critical computations}
Before moving on to the critical computation (Theorem \ref{thm:critical}) let us prove that the critical case caracterized by the equality in the second display of Theorem \ref{thm:phase} actually corresponds to the natural fact that one "cannot increase the number of cars" and stay subcritical:

\begin{lemma}[Criticality] Suppose $(\star)$. Then we have equality in \eqref{eq:threshold} iff for any $ \varepsilon>0$ the law with generating function $G_{ \varepsilon} (t)= G (t) + \varepsilon t - \varepsilon$ is supercritical.
\end{lemma}
\proof  
Suppose that $ \mu$ is subcritical in the sense of Theorem \ref{thm:phase} and look at the probability measure $ \mu_{ \varepsilon}$ such that its generating function is given by $ G_{ \varepsilon} (t)= G (t) + \varepsilon t - \varepsilon$ for some $ \varepsilon >0$. First notice that $ \mu_{ \varepsilon}$ satisfies Assumption ($\star$) for $ \varepsilon$ small enough. Indeed the radius of convergence of $ G_{\varepsilon}$ is that of $G$ and the quantity
\begin{eqnarray*}2 (G_{ \varepsilon}(t) - t G_{ \varepsilon}'(t))^2 -  t^2 G_{ \varepsilon}(t) G_{ \varepsilon}''(t) &=& (2 (G(t) - t G'(t) - \varepsilon)^2 -  t^2 (G(t) + \varepsilon(t-1))G''(t) \\
&=&2 (G(t) - t G'(t))^2  -  t^2 G(t)G''(t) \\
&\quad& \qquad+ 2 \varepsilon^2  - 2\varepsilon (G(t) - t G'(t)) -  \varepsilon t^2 (t-1)G''(t)
\end{eqnarray*}
is negative at $ t_c$ when $ \varepsilon$ is small enough, so that $ t_c^{ \varepsilon} := \min \{ t \geq 0,\ 2 (G_{ \varepsilon}(t) - t G_{ \varepsilon}'(t))^2 =  t^2 G_{ \varepsilon}(t) G_{ \varepsilon}''(t) \} < t_c$ and the function $ \varepsilon \mapsto t_{c}^{ \varepsilon}$ is continuous in a positive neighborhood of $0$. To determine whether $ \mu_{ \varepsilon}$ is subcritical or not, we then need to determine the sign of 
$$ (t_c^{ \varepsilon}-2)G_{ \varepsilon}(t_c^{ \varepsilon}) - t_c^{ \varepsilon}(t_c^{ \varepsilon}-1)G_{ \varepsilon}'(t_c^{ \varepsilon}), $$
which is then continuous is $ \varepsilon$. Thus if Equation \eqref{eq:threshold} is not an equality, we can increase $ \mu$ and remain subcritical. 

Suppose now that  \eqref{eq:threshold} is an equality. We can show that 
$$ \partial_{ y} \varphi_{ \varepsilon} (t_c - \delta)  = \frac{6 G(t_c)+ t_c^2 G^{(3)}(t_c) (t_c - 1)^3 }{4 G(t_c)^2(t_c - 1)^3 } \delta - \frac{3}{2 (t_c-1) G(t_c)^2} \varepsilon + o(\varepsilon) + o(\delta). $$
When $t_c - \delta = t_c^{\varepsilon}$, then the left hand side is zero so that the main asymptotic on the right hand side should be $0$. Hence $ \delta = t_c - t_c^{ \varepsilon}$ is of order $\varepsilon.$
Moreover, 
$$(t_c^{ \varepsilon}-2)G_{ \varepsilon}(t_c^{ \varepsilon}) - t_c^{ \varepsilon}(t_c^{ \varepsilon}-1)G_{ \varepsilon}'(t_c^{ \varepsilon}) = -2 \varepsilon ( t_c^{ \varepsilon}-1 ) + O( (t_c^{ \varepsilon} - t_c)^2). $$
is negative when $ \varepsilon$ small enough and $ \mu_{ \varepsilon}$ is supercritical by Theorem \ref{thm:phase}. \endproof

\begin{proof}[Proof of Theorem \ref{thm:critical}]. Let us focus now on the case when \eqref{eq:threshold} is an equality. In that case, since \eqref{eq:fundabis} is an equality for $x=x_c$, then $x_c = p_{ \circ}.$  If $x$ satisfies Equation \eqref{eq:xfcty} and the parking is critical, so that $y =t_c$ is solution of $(t-2)G(t) = t(t-1)G'(t)$, then \begin{equation}\label{eq:xf0crit}x = x_c = \frac{t_c^2}{4(t_c-1) G(t_c)} \qquad \mbox{and} \qquad \mathbf{F}_0(x_c)^2 = \frac{1}{p_{ \circ} G(0)} =  \frac{1}{p_{ \circ} \mu_0}.\end{equation}
Moreover, using Equation \eqref{eq:funda2}, we obtain 
$$p_{ \bullet} = \sqrt{ \frac{p_{ \circ}}{ \mu_0}}- p_{ \circ},$$
and this concludes the proof of Theorem \ref{thm:critical}.
\end{proof}
%To compute the expectation of $ X$, the simplest way is to take the expectation in the recursive Equation \eqref{eq:recursive}. Recall that $ p_{ \circ} = \mathbb{P}\left( X=0\right)$ and writing $(X-1)_{+} = X-1 + \mathbb{1}_{X=0}$, we obtain 
%$$ \mathbb{E}[X] = 2(1-p_{ \circ}) - G'(1) <2.$$
%We can also recover this result by using the fact that $yp_{ \circ} F (p_{ \circ},y)$ is the generating function of $X$ (see Proposition  \ref{prop:noblack} and/or Equation \eqref{eq:funda1}) and derivating with respect to $y$ the expression of $ \mathbf{F}(x,y)$ which we obtained by solving Tutte's equation. 
%$$ \partial_{y} (yp_{ \circ} F) (p_{ \circ},1) = \frac{(2 - 6 x_c + 4 \mathbf{F}_0(x_c)^2 x_c^2 G(0) - G'(1) + 
% 2 x_c G'(1) - 2 \mathbf{F}_0(x_c)^2 x_c^2 G(0) G'(1) - 
% \sqrt{1 + 4 x (-1 + \mathbf{F}_0(x_c)^2 x G(0))} (-2 + 2 x + 
%    G'(1)))}{2 \sqrt{1 + 4 x (-1 + \mathbf{F}_0(x_c)^2 x G(0))}}$$
%Simplifying this expression using Equations \eqref{eq:xf0crit} we obtain 
%$$ \mathbb{E}[X] = 2(1-p_{ \circ}) - G'(1) <2.$$

%\alice{Il y a tout pour le thm 2, non? Peut-etre le rayon de cv...} \nico{C'etait quoi le delire sur les egalites dans les RCV au fait, j'ai un peu oublie}\alice{Je sais pas si je comprends mais il me semble qu'on se demandait quand est-ce que $RCV_{x} = RCV_{\mu}$ i.e. $t_c = RCV_{G}$}

\begin{remark} \label{ref:p0} The fact that $ p_{ \circ} >1/2$ is a consequence of the equation for $t$ given by Theorem \ref{thm:phase}. Indeed, in the critical case, then $ p_{ \circ} = x_c = t_c^2/(4(t_c-1) G(t_c))$ where $t_c$ is given by Theorem \ref{thm:phase}. But since $ t_c > 1$, we have $ G(t_c) - G(1) \leq (t_c- 1)G'(t_c) = (t_c-2)G(t_c)/t_c$. Thus $G(t_c) \leq t_c/2$ and
 $$ p_{ \circ} \geq \frac{2t_c}{t_c-1} > 1/2.$$
\end{remark}
 
% \begin{remark}[About the radii of convergence] \ref{ref:radius2}
% \end{remark}
 
 %\alice{A bouger je ne sais pas ou}
%Calculs explicite de la loi taille, puis espÃÂÃÂ\c cÃÂÃÂ\'erance de la taille, du flux...
%\subsection{Geometry of critical clusters}
%Analyse de singularitÃÂÃÂ\c cÃÂÃÂ\'e -> exposant critique (5/2) pour la taille du cluster non. Blablba habituel sur les GFT.
%-> mais en revanche dÃÂÃÂ\c cÃÂÃÂ\'ecroissance exponentielle pour le flux? Universel ou pas ?
%\subsection{}

\subsection{Examples}
Let us proceed to the computation of the critical threshold for parking for various families of stochastically  increasing laws. In the first four cases below, it is easy to check that condition $(\star)$ holds so that we can just apply the general formulas. In the last example we explain how our techniques can be applied even if $(\star)$ does not hold. 
%\alice{En fait, comme on peut pas faire les calculs, faut surement plutÃÂÃÅt mettre les exemple aprÃÂÃÂ¡s...}
%In the general case, it is not so easy to compute explicitly $ Y(x)$ for $x \in [0,x_c].$ However, we can use equation \eqref{eq:phixc} to obtain $ Y(x_c)$, and then deduce with Equation \eqref{eq:xfcty} the radius of convergence $x_c$ of $T_1$ and with Equation \eqref{eq:F1} the value of $T_1$ at $x_c.$ \nico{Le faire dans la section 5.1 et mettre les exemples apres?}
\paragraph{Binary$_{0/2}$ car arrivals.} As a first example, we can imagine that either $0$ or $2$ cars arrive at each spot, i.e.\ the law of the car arrivals is $ \mu = (1- \frac{ \alpha}{2}) \delta_0 + \frac{ \alpha}{2} \delta_2$, so that $ G(t) = (1- \frac{ \alpha}{2}) +  \frac{ \alpha}{2}t^2$. This is the example considered in  \cite[Proposition 3.5]{GP19} and \cite[Proposition 4]{BBJ19}. Explicit computations show in this case that

$$ t_{c}=Y(x_c) = \sqrt{ \frac{2- \alpha}{3 \alpha}}, \qquad x_c = \frac{3 \sqrt{3}}{16 \sqrt{ \alpha(2- \alpha)}} \qquad \mbox{and} \qquad \mathbf{F}_0(x_c)= \frac{4\sqrt{6}}{9}\cdot $$
Note that $ \mathbf{F}_{0} (x_c)$ does not depend on $ \alpha$, see the remark below. 
We then see that for $t = t_c$, the Inequality \eqref{eq:threshold} is quadratic in $ \alpha$ and is satisfied as soon as  $ \alpha < \alpha_c = 1/14$. 
%\nico{bizarre que $F_{0}(x_{c})$ ne depende pas de $p$.... Ah mais oui c'etait ta remarque... il faut l'expliquer alors, j'essaie en bas} \nico{et donc c'est critique quand ???} \nico{et le cas 0/k a inclure}.
\paragraph{Binary$_{0/k}$ car arrivals.} In the case when $ \mu = (1- \frac{ \alpha}{k}) \delta_0 + \frac{ \alpha}{k} \delta_k$, so that $ G(t) = (1- \frac{ \alpha}{k}) +  \frac{ \alpha}{k}t^k$ with $k \geq 3$, we have
$$ Y(x_c) = \left(\frac{ (k - a) (-4 + k (3 + k - \sqrt{(k-1) (k+7)})}{2a ( k-2) ( k-1)}\right)^{1/k} , $$ 
 $$ \mathbf{F}_0 (x_c) = \frac{\sqrt{2} (3 k - \sqrt{(k-1) (k+7)}-3 )}{( k-2) (k-1) \sqrt{\frac{5 - k + \sqrt{(k-1) ( k+7)}}{(k-1) k}}}. $$
so that the model is critical at $$ \alpha_c ( \mathrm{Binary}_{0/k}) = \frac{k}{1 + 2^{-k-2} \left( 3 + \sqrt{ \frac{k+7}{k-1}}\right)^k \left( (k-1)(k+4)+ k \sqrt{(k+7)(k-1)}\right)} . $$
\paragraph{Poisson.} Suppose the law of the car arrivals is Poisson with mean $ \alpha> 0$, so that in this case $G(t) = e^{ \alpha(t - 1)}. $ Again, explicit computation show that
%\alice{On ne peut pas inverser mais $T_1(x) =-1 + \frac{2 \mathrm{e}^{ \frac{a Y(x) }{2}} \sqrt{1 - a Y(x)}}{2 - a Y(x)}.$ RCV quand $ Y(x)= (2 - \sqrt{2})/a$, i.e. A verifier }

$$ Y(x_c) = \frac{2- \sqrt{2}}{ \alpha}, \qquad x_c = \frac{(\sqrt{2}-1) \mathrm{e}^{ \alpha-2 + \sqrt{2} }}{2 \alpha} \qquad \mbox{and} \qquad \mathbf{F}_0(x_c) =  \sqrt{2 ( \sqrt{2}-1)} \mathrm{e}^{1 - 1/\sqrt{2}}, $$
so that the model is subcritical for parking as long as $\alpha \leq \alpha_c $ with $$ \alpha_c = 3 - 2 \sqrt{2}.$$

\paragraph{Geometric.} Consider here the case when $ G(t) = 1/(1+ \alpha- \alpha t)$. Then
$$ Y(x_c) = \frac{1+ \alpha}{3 \alpha}, \qquad x_c = \frac{(1+ \alpha)^2}{12 \alpha} \qquad \mbox{and} \qquad \mathbf{F}_0(x_c) = \frac{2}{ \sqrt{3}} \cdot$$
so that the model is subcritical for parking as long as $ \alpha \leq \alpha_c $ with $ \alpha_c = 1/8.$

\begin{remark}[Combinatorial counting] 
The reader may be puzzled by the fact that in the last four cases the value $ \mathbf{F}_{0}( x_{c})$ does not depend on the parameter $\alpha$. 
This is because in each case, the dependence on $\alpha$ of the  $\mu_{\alpha}$-weight of a fully parked tree of size $n$ is of the form $(c_{\alpha})^{n}$, for a constant $c_{\alpha}$ depending on $\alpha$ that cancels out at criticality. To wit, consider fully-parked trees associated with $0/k$ arrivals : in this case, for $n$ a multiple of $k$, fully parked trees of size $n$ have weight $\prod_{v} \mu_{a(v)} = \mu_{0}^{n/k} \mu_{k}^{(1-1/k)n}  =(\mu_{0}^{1/k} \mu_{k}^{(1-1/k)})^n$.

%The reader may be puzzled by the fact that in the last three cases the value $ \mathbf{F}_{0}( x_{c})$ does not depend on the parameter $\alpha$.  This is because in each case,  the quantity $\mathbf{F}_0(x_c)$ can be expressed as a sum over ``combinatorial'' fully-parked trees:  in the case of $0/k$ arrivals they correspond to binary trees of size $nk$ decorated with exactly $n$ vertices having an arrival of $k$ cars so that they all can park, in the case of the Poisson distribution they correspond to decoration with numbered cars, and in the case of the Geometric distribution the cars are unlabeled.  In each case, the $\mu_{\alpha}$-weight of such a tree of size $n$ is proportional to $c_{\alpha}^{n}$ so that at the critical point these exponential weight disappear. \nico{oui c'est tres mal dit...}
\end{remark}

%RCV valuers au RCV.

\paragraph{Without $ (\star)$.} \label{sec:non-generic} When hypothesis $(\star)$ is not satisfied we can still apply our method: If the generating series $G$ has a radius of convergence $y_{c}$ then the value $t_{c}$ is then replaced by 
$$\tilde{t} _{c} = \min\left\{ y_{c} , \min \{ t \geq 0:\ 2(G(t)- t G'(t))^2 = t^2 G(t) G''(t) \}\right\},$$
and we need to check that $\tilde{x}_{c}$ defined analogously by \eqref{eq:defxc} is again the radius of convergence of the series $ \mathbf{F}_{0}$ (we did not try to prove such a general statement here). Then using Proposition \ref{prop:criti}, the subcriticality of the parking is equivalent to the fact that $ \mu_{0} \tilde{x}_{c}(  \mathbf{F}_{0}(\tilde{x}_{c}))^{2} \geq 1$, the only different point is that we cannot use the equality $2(G(\tilde{t}_{c})- \tilde{t}_{c} G'(\tilde{t}_{c}))^2 = \tilde{t}_{c}^2 G(\tilde{t}_{c}) G''(\tilde{t}_{c})$ to further simplify the expression. As an example of such law, consider the generating function 
  \begin{eqnarray} \label{eq:ng} G(t) = 1+ \frac{1+ t^2}{26} - \frac{1}{13} \left( \frac{3-t}{2}\right)^{7/3}.  \end{eqnarray}
The radius of convergence of $G$ is equal to $3$ but $G'(3)$ and $G''(3)$ exist. An explicit computation shows that $(\star)$ holds with $t_{c}=y_{c}= \tilde{t}_{c}=3$ for $G$, and furthermore $G$ is critical for the parking process.  However if one considers $\tilde{G} = 0.9 + 0.1G$ then $\tilde{t}_{c} = 3$ and $(\star)$ does not hold but still $ \tilde{G}$ is subcritical for the parking process.

\section{Extensions and comments}

\label{sec:comments}
In this work, we voluntarily stick to the simplest case of the binary tree with i.i.d.\ arrivals without specific conditions to keep the paper accessible to a wide audience. Let us mention a few perspectives that our approach opens:
\subsection{Non-generic and dense  case} In this work, we focused on the localization of the threshold and on the computation of some critical quantities. One could also try to get scaling limits of critical components and compute several critical or near-critical exponents. As mentioned in the introduction, we expect that  a large family of critical car arrivals (say, with bounded support) belong to a common universality class where we expect that the cluster size of the root has a tail that decays as $ n^{-5/2}$ as $ n \to \infty$ and where the scaling limits of the critical components are given by $3/2$-stable Growth-Fragmentation trees. But when the car arrivals have a heavy tail (and when the parameters are fine-tuned so that the law is critical), we hope to see different universality classes. Actually, as seen in Section \ref{sec:Y(x)}, the singular behavior of $ \mathbf{F}_{0}$ near its radius of convergence is linked to the behavior of the $Y(x)$ near near $x_{c}$, see Figure \ref{fig:scenarii}.  For instance, in the example \eqref{eq:ng} an explicit computation shows that the singular behavior of $ \mathbf{F}_{0}(3-x) - \mathbf{F}_0 (3)$ is of the form $C x^{4/3}$  which indicates a polynomial decay of the critical cluster with exponent $n^{-7/3}$. This is very similar to the scenarios that happened in the enumeration of plane fully parked trees in Chen \cite{chen2021enumeration}, with the notable difference that in our model the dense case can be critical for the parking process. %. Let us explain this phenomenon heuristically (we are not claiming rigorous results below) in the behavior of the change of variable $x \to Y(x)$ near $x_{c}$, see Figure \ref{fig:ng} \begin{itemize} \item Under generic conditions, e.g. under our assumption $(\star)$ and if $G$ has a radius of convergence strictly larger than $t_{c}$ it follows from the proof of Lemma \ref{lem:Y'} that $Y(t_{c}- \varepsilon)$ behaves as $ \sqrt{ \varepsilon}$ and this creates a singularity on $ \mathbf{F}_{0}$ of the form $ \mathbf{F}_{0}(x_{c} - \varepsilon) = \mathrm{regular part} + C \varepsilon^{3/2}$ as $ \varepsilon \to 0$. By singularity analysis this creates the polynomial behavior $[x^{n}]  \mathbf{F}_{0}(x) \sim C x_{c}^{-n} n^{5/2}$ mentioned just above.  \item However, if $G$ has radius of convergence $t_{c}$ one might cook-up examples where $Y(t_{c}- \varepsilon) \approx \varepsilon^{\alpha}$ for $\alpha \ne \frac{1}{2}$. One such example is given by \eqref{eq:ng} where $G$ is critical and where an explicit computation shows that $$ \mathbf{F}_{0}(3-x) - \mathbf{F}_0 (3) \sim C x^{4/3}.$$ \alice{La je sais plus... la somme fait $1$ ou $-1$ non?} \nico{y'a pas de partie en $\varepsilon$ ?} which indicates a polynomial decays of the critical cluster with exponent $n^{-7/3}$. \end{itemize} \begin{figure}[!h] \begin{center  \includegraphics[width=15cm]{non-generic}  \caption{Illustration of the scenarii  for the change of variable $ x \mapsto Y(x)$. On the left, when $t_{c}$ is strictly less than the radius of convergence of $G$, we expect a singularity of type $(x_{c}-x)^{3/2}$.  \nico{Je crois que dans notre cas, le cas dense peut se retrouver  la fois dans le sub critique et supercritique}}  \end{center}  \end{figure}
See also \cite[Theorem 1.2]{chen2019derrida} for the Derrida-Retaux model with heavy-tailed distributions where the free-energy has a peculiar behavior. We plan on studying those different behaviors in forthcoming works.

%On peut faire un exemple avec $RCV_{X} = RCV_{\mu}$ et o\`u ca p\`ete bizarrement \`a ce point.
%D'ailleurs conditions sur $G$ pour $RCV_{X} = RCV_{\mu}$. 

\subsection{General case of $d$-ary tree and GW trees}
Our work may be extended to parking on more general trees such as $d$-ary trees and perhaps supercritical Bienaym\'e--Galton--Watson trees. The crux is of course the enumeration of fully parked trees. In the case of supercritical Bienaym\'e--Galton--Watson trees, one would probably need the addition of another variable $z$ counting the number of adjacent vertices of the fully parked tree inside the global tree (in our case, we had a fixed number $n+1$ of  vertices adjacent from above to a fully parked tree of $ \mathbb{B}$ of size $n$). We wonder whether the randomness of the underlying tree may yield to different universality classes compared to the case of $d$-ary trees. 

\subsection{Links with Derrida-Retaux model} As mentioned several times in the paper, the Derrida-Retaux model is closely related to the parking process on $ \mathbb{B}$. We wonder whether a firm connection can be made between the two models.

%\label{sec:non-generic}
%% -> vous ÃÂÃÂ\c cÃÂÃÂÃÂÃÂtes bons, mais pas courageux}
%\subsection{Locality and non triviality of $t_{ \mathrm{crit}}$} 
%En fait, je ne suis meme pas sur de le mettre dans ce papier, ca brouille le message.
\bibliographystyle{siam}
%\bibliography{/Users/nicolascurien/Dropbox/Macros-Bibli/bibli}
%\bibliography{/Users/olivier/Dropbox/22.ParkingGeneral/bibli.bib}
\bibliography{/Users/contat/Dropbox/Articles/22.ParkingGeneral/bibli.bib}

\begin{thebibliography}{10}

\bibitem{BBJ19}
{\sc R.~Bahl, P.~Barnet, and M.~Junge}, {\em Parking on supercritical
  {G}alton-{W}atson trees}, arXiv:1912.13062,  (2019).

\bibitem{BBCK18}
{\sc J.~Bertoin, T.~Budd, N.~Curien, and I.~Kortchemski}, {\em Martingales in
  self-similar growth-fragmentations and their connections with random planar
  maps}, Probab. Theory Related Fields, 172 (2018), pp.~663--724.

\bibitem{BCK18}
{\sc J.~Bertoin, N.~Curien, and I.~Kortchemski}, {\em Random planar maps and
  growth-fragmentations}, Ann. Probab., 46 (2018), pp.~207--260.

\bibitem{BMJ06}
{\sc M.~Bousquet-M{\'e}lou and A.~Jehanne}, {\em Polynomial equations with one
  catalytic variable, algebraic series and map enumeration}, J. Combin. Theory
  Ser. B, 96 (2006), pp.~623--672.

\bibitem{chassaing2002phase}
{\sc P.~Chassaing and G.~Louchard}, {\em Phase transition for parking blocks,
  {B}rownian excursion and coalescence}, Random Structures \& Algorithms, 21
  (2002), pp.~76--119.

\bibitem{chen2021enumeration}
{\sc L.~Chen}, {\em Enumeration of fully parked trees}, arXiv preprint
  arXiv:2103.15770,  (2021).

\bibitem{chen2019derrida}
{\sc X.~Chen, V.~Dagard, B.~Derrida, Y.~Hu, M.~Lifshits, and Z.~Shi}, {\em The
  {D}errida--{R}etaux conjecture on recursive models}, The Annals of
  Probability, 49 (2019), pp.~637--670.

\bibitem{contat2020sharpness}
{\sc A.~Contat}, {\em Sharpness of the phase transition for parking on random
  trees}, Random Structures \& Algorithms,  (2020).

\bibitem{contat2022last}
\leavevmode\vrule height 2pt depth -1.6pt width 23pt, {\em Last car
  decomposition of planar maps}, arXiv preprint arXiv:2205.10285,  (2022).

\bibitem{ConCurParking}
{\sc A.~Contat and N.~Curien}, {\em Parking on {C}ayley trees \& frozen
  {E}rd\"os-{R}\'enyi}, arXiv:2107.02116.

\bibitem{CurStFlour}
{\sc N.~Curien}, {\em Peeling random planar maps, Saint-Flour course 2019},
  https://www.imo.universite-paris-saclay.fr/$\sim$curien/.

\bibitem{CH19}
{\sc N.~Curien and O.~H\'enard}, {\em The phase transition for parking on
  {G}alton-{W}atson trees}, Discrete Analysis,  (2022).

\bibitem{Flajolet:analytic}
{\sc P.~Flajolet and R.~Sedgewick}, {\em Analytic combinatorics}, Cambridge
  University Press, 2009.

\bibitem{GP19}
{\sc C.~Goldschmidt and M.~Przykucki}, {\em Parking on a random tree},
  Combinatorics, Probability and Computing, 28 (2019), pp.~23--45.

\bibitem{king2019prime}
{\sc W.~King and C.~H. Yan}, {\em Prime parking functions on rooted trees},
  Journal of Combinatorial Theory, Series A, 168 (2019), pp.~1--25.

\bibitem{konheim1966occupancy}
{\sc A.~G. Konheim and B.~Weiss}, {\em An occupancy discipline and
  applications}, SIAM Journal on Applied Mathematics, 14 (1966),
  pp.~1266--1274.

\bibitem{LaP16}
{\sc M.-L. Lackner and A.~Panholzer}, {\em Parking functions for mappings},
  Journal of Combinatorial Theory, Series A, 142 (2016), pp.~1 -- 28.

\bibitem{le2020growth}
{\sc J.-F. Le~Gall and A.~Riera}, {\em Growth-fragmentation processes in
  {B}rownian motion indexed by the {B}rownian tree}, Annals of Probability, 48
  (2020), pp.~1742--1784.

\bibitem{Tut62}
{\sc W.~T. Tutte}, {\em A census of planar triangulations}, Canad. J. Math., 14
  (1962), pp.~21--38.

\end{thebibliography}
\end{document}